\documentclass[11pt]{article}

\usepackage{latexsym}
\usepackage{amssymb}
\usepackage{amsthm}
\usepackage{amscd}

\usepackage{amsmath}

\usepackage{mathrsfs}
\usepackage[all]{xy}
\usepackage{hyperref} 
\usepackage[usenames,dvipsnames]{color}
\usepackage{graphicx,eepic}
\usepackage{float}

\usepackage{color}

\theoremstyle{definition}
\newtheorem* {theorem*}{Theorem}
\newtheorem{theorem}{Theorem}[section]
\newtheorem{thmdef}[theorem]{Theorem-Definition}

\theoremstyle{definition}
\newtheorem{observation}[theorem]{Observation}

\newtheorem* {example*}{Example}

\newtheorem{lemma}[theorem]{Lemma}
\theoremstyle{definition}
\newtheorem{definition}[theorem]{Definition}
\theoremstyle{definition}
\newtheorem* {notation}{Notation}

\newtheorem{proposition}[theorem]{Proposition}
\newtheorem{corollary}[theorem]{Corollary}

\newtheorem* {remark}{Remark}
\theoremstyle{definition}
\newtheorem {example}[theorem]{Example}
\theoremstyle{definition}

\theoremstyle{definition}

\theoremstyle{definition}

\xyoption{dvips}

\usepackage{fullpage}

\numberwithin{equation}{section}

\def\op{\mathrm{op}}
\def\modu{\ (\mathrm{mod}\ }

\def\({\left(}
\def\){\right)}
       \newcommand{\QQ}{\mathbb{Q}}    \newcommand{\cA}{\mathcal{A}}
   
   \newcommand{\cI}{\mathcal{I}}
\newcommand{\cJ}{\mathcal{J}} 

\newcommand{\cB}{\mathcal{B}}

\def\NN{\mathbb{N}}
  \def\Hom{\mathrm{Hom}} \def\End{\mathrm{End}} \def\ZZ{\mathbb{Z}} \def\Aut{\mathrm{Aut}}  
         \def\spanning{\textnormal{-span}}   \def\cB{\mathcal{B}}
  \def\wt{\widetilde}

\newcommand{\cM}{\mathcal{M}}

\newcommand{\cL}{\mathcal{L}}

\newcommand{\one}{{1\hspace{-.11cm} 1}}

\newcommand{\sgn}{\mathrm{sgn}}

\def\barr{\begin{array}}
\def\earr{\end{array}}
\def\ba{\begin{aligned}}
\def\ea{\end{aligned}}
\def\be{\begin{equation}}
\def\ee{\end{equation}}

\def\sC{\mathscr{C}}

\def\qquand{\qquad\text{and}\qquad}

\def\I{\mathbf{I}_*}

\def\act{\ltimes}
\def\ds{\displaystyle}

\def\H{\mathcal H}
\def\cM{\mathcal M}

\def\I{\mathbf{I}}

\def\ben{\begin{enumerate}}
\def\een{\end{enumerate}}

\newcommand{\genrep}[4]{{\left[\barr{l  l} #1 & #2 \\ #3 & #4 \earr\right]}}
\newcommand{\menrep}[8]{{\left[\barr{l  l} #1 & #2 \\ #3 & #4 \\ #5 & #6 \\ #7 & #8 \earr\right]}}

\makeatletter
\renewcommand{\@makefnmark}{\mbox{\textsuperscript{}}}
\makeatother

\allowdisplaybreaks[1]

\UseCrayolaColors

\begin{document}
\title{Comparing and characterizing some constructions of canonical bases  from Coxeter systems}
\author{Eric Marberg\footnote{This research was conducted with support from the National Science Foundation.}
\\
Department of Mathematics \\
Stanford University \\
{\tt emarberg@stanford.edu}}

\date{}

\maketitle

\begin{center}
Dedicated to David Vogan on the occasion of his 60th birthday.
\end{center}

\begin{abstract}
The Iwahori-Hecke algebra $\H$ of a Coxeter system $(W,S)$
has a ``standard basis'' indexed by the elements of $W$ and a ``bar involution'' given by a certain antilinear map. Together, these form an example of what Webster calls  
a pre-canonical structure, relative to which the well-known Kazhdan-Lusztig basis of $\H$ is a canonical basis. Lusztig and Vogan have  defined a representation of a modified Iwahori-Hecke algebra on the free $\ZZ[v,v^{-1}]$-module generated by the set of twisted involutions in $W$, and shown that this module has a unique  pre-canonical structure compatible with the $\H$-module structure, which admits its own  canonical basis which can be viewed as a generalization of the Kazhdan-Lusztig basis. 
One can modify the definition of Lusztig and Vogan's module to obtain other pre-canonical structures, each of which admits a unique canonical basis indexed by twisted involutions. We classify all of the pre-canonical structures which arise in this manner, and explain the relationships between their resulting canonical bases. Some of these canonical bases are equivalent in a trivial fashion to Lusztig and Vogan's construction, while others appear to be unrelated. Along the way, we also  clarify the differences between Webster's notion of a canonical basis and the related concepts of an IC basis and a $P$-kernel. 
%
\end{abstract}

\setcounter{tocdepth}{2}
\tableofcontents

\section{Introduction}

Let $(W,S)$ be a Coxeter system and write $\H$ for its associated Iwahori-Hecke algebra.
 This algebra has a ``standard basis'' indexed by the elements of $W$, whose structure constants have a simple inductive formula. The Kazhdan-Lusztig basis of $\H$ is the unique basis which is invariant under a certain antilinear map $\H \to \H$, referred to as the ``bar involution,'' and whose elements are each unitriangular linear combinations of standard basis elements with respect to the Bruhat order.
 The standard basis and bar involution of $\H$ are an example of what Webster \cite{Webster} calls  
a \emph{pre-canonical structure}, relative to which the Kazhdan-Lusztig basis is a \emph{canonical basis}.
This terminology, whose precise definition we review in Section \ref{canon-sect}, is useful for organizing several similar constructions attached to Coxeter systems.
 Webster's idea of a canonical basis is closely related to Du's notion of an \emph{IC basis} \cite{Du} and also to Stanley's notion of a \emph{$P$-kernel} \cite{StanleyPKernel}, and in Section \ref{IC-sect} we discuss the relationship between these three concepts.

In \cite{LV2,LV1,LV6}, Lusztig and Vogan study a representation of a modified Iwahori-Hecke algebra $\H_2$ on the free $\ZZ[v,v^{-1}]$-module generated by the set of twisted involutions $\I = \I(W,S)$ in a Coxeter group. (See Section \ref{twisted-sect} for the definition of this set; though we mean something more general, in this introduction one can simply take $\I = \{ w \in W: w^2=1\}$.) They show that this module has a unique pre-canonical structure which is compatible with the action of 
$\H_2$,
and that this structure admits a canonical basis, of which the Kazhdan-Lusztig basis can be viewed as a special case.

The definition of  Lusztig and Vogan's $\H_2$-representation has a particularly simple form, and gives an example of a \emph{generic $(\H_2,\I)$-structure} as defined in Section \ref{32-sect}.
It turns out that there are a number of slight modifications one can make to this definition which produce other $\H_2$-module structures 
on the free $\ZZ[v,v^{-1}]$-algebra generated by $\I$; some (but not all) of 
 these modules likewise possess a unique pre-canonical structure compatible with the action of $\H_2$; in each such case there is  a unique associated canonical basis. 
We review Lusztig and Vogan's results in Section \ref{lv-sect}, and derive from them a family of analogous theorems (along the lines just described) in Section \ref{precanon3-sect}.
 In Section \ref{precanon2-sect} we present another 
variation of these results, in which
the role of the modified Iwahori-Hecke algebra $\H_2$ is replaced by the usual algebra $\H$. 
These constructions  give three canonical bases indexed by the twisted involutions in a Coxeter group; these bases all can be seen as generalizations of the Kazhdan-Lusztig basis of $\H$, but, somewhat unexpectedly, they do not appear to be related to each other in any simple way.

In Sections \ref{16-sect} and \ref{32-sect} we describe a precise sense in which these three bases account for all canonical bases on this space.
Specifically, 
we define in Section \ref{morph-sect}  a category whose objects are pre-canonical structures on free $\ZZ[v,v^{-1}]$-modules. Our definition of morphisms in this category has the following appealing properties:
\ben
\item[(i)]  Canonical bases arising from  isomorphic  pre-canonical structures are always related in a simple way; in particular, their coefficients (when written as sums of standard basis elements) are equal up to a change of sign or the variable substitution $v \mapsto -v$; see Corollary \ref{canon-cor}.

\item[(ii)] Assume the free $\ZZ[v,v^{-1}]$-module generated by $W$ has a pre-canonical structure in which the natural basis $W$ is standard. If this structure satisfies a natural compatibility condition with an $\H$-representation on the ambient space, then it is isomorphic to the pre-canonical structure on $\H$ itself, and so it has a unique canonical basis which can be identified in the sense of (i) with the Kazhdan-Lusztig basis; see Theorem \ref{4-thm}.

\een
With respect to these definitions, our main results are as follows.
Suppose we are given a pre-canonical structure on the free $\ZZ[v,v^{-1}]$-module generated by the set of twisted involutions in $W$, in which the natural basis $\I$ is the standard one.
We prove that
\ben
\item[(1)] If the structure is compatible with any representation of $\H$ of a certain natural form, then it is isomorphic to the pre-canonical structure we define in Section \ref{precanon2-sect}; see Theorem \ref{16-thm}.

\item[(2)] If the structure is compatible (in a certain natural sense) with a representation of the modified Iwahori-Hecke algebra $\H_2$, then it is isomorphic to one of four pre-canonical structures:  the one Lusztig and Vogan define in \cite{LV2,LV1},
the one we define in Section \ref{precanon3-sect},
or one of two non-isomorphic structures derived from the one  given in Section \ref{precanon2-sect}; 
see Theorem \ref{32-thm}.

\een
These results provide some formal justification for considering the  pre-canonical structures described in Sections  \ref{lv-sect}, \ref{precanon3-sect}, and \ref{precanon2-sect} 
to be particularly natural objects.
Lusztig and Vogan have given two interpretations of the first structure, in terms of the geometry of an associated algebraic group when $W$ is a Weyl group \cite{LV1} and in terms of the theory of Soergel bimodules for general $W$ \cite{LV6}. It remains an open problem to give similar interpretations of the two other  pre-canonical structures.

\subsection*{Acknowledgements}

I thank Daniel Bump, Persi Diaconis, Richard Green, George Lusztig, David Vogan, and Zhiwei Yun for helpful discussions related to the development of this paper.

\section{Preliminaries}

\subsection{Canonical bases}
\label{canon-sect}

Throughout we let $\cA = \ZZ[v,v^{-1}]$ denote the ring of Laurent polynomials with integer coefficients  in a single indeterminant. 
We write $f \mapsto \overline {f}$ for the ring involution of $\cA$ with $v \mapsto v^{-1}$
 and say that a map $\varphi : U \to V$ between $\cA$-modules is \emph{$\cA$-antilinear} if $\varphi(fu) = \overline{f}\cdot \varphi(u)$ for $f \in \cA$ and $u \in U$.
 Let $V$ be a free $\cA$-module. 
 
 \begin{definition}\label{precanon-def}
A \emph{(balanced) pre-canonical structure} on $V$ consists of
\begin{itemize}
\item a ``bar involution''   $\psi $ given by an $\cA$-antilinear map $V \to V$
with $\psi^2=1$. 

\item a ``standard basis'' $\{a_c\}$ with partially ordered index set $(C,\leq)$ such that 
\[ \psi(a_c) \in a_c + \sum_{c' < c} \cA\cdot a_{c'}.\]
\end{itemize}
\end{definition}

This  is equivalent to Webster's definition of a \emph{balanced pre-canonical structure} \cite[Definition 1.5]{Webster}. In this work we will only consider  pre-canonical structures  which are balanced in this sense, and from this point on we  drop the adjective ``balanced'' and just refer to ``pre-canonical structures.'' The reader should note, however, that in \cite{Webster} a pre-canonical structure refers to a slightly more general construction which includes Definition \ref{precanon-def} as a special case.

Assume $V$ has a pre-canonical structure $(\psi,\{a_c\})$; we then  have this accompanying notion.

\begin{definition}\label{canon-def}
A set of vectors $\{b_c\}$ in $V$ also indexed by $(C,\leq)$ is a  \emph{canonical basis} 
if 
\begin{itemize}
\item[(C1)] each vector $b_c$ in the basis  is invariant under $\psi$.
\item[(C2)] each vector $b_c$ in the basis  is in the set $b_c = a_c + \sum_{c' < c} v^{-1}\ZZ[v^{-1}] \cdot a_{c'}$.
\end{itemize}
\end{definition}


This definition of a canonical basis is slightly different from the one which Webster gives  \cite[Definition 1.7]{Webster}, but is equivalent when the pre-canonical structure on $V$ is balanced (which we assume everywhere in this work) by \cite[Lemma 1.8]{Webster}.


\begin{example}
We view the ring $\cA$ itself as  possessing the pre-canonical structure in which  the bar involution  is the map $f \mapsto \overline f$
and the standard basis is the singleton set $\{1\}$.
This structure admits a canonical basis, which is again just $\{1\}$.
\end{example}

 The following crucial property of a canonical basis appears in  the introduction of \cite{Webster};
its elementary proof is an instructive exercise.

\begin{proposition}[Webster \cite{Webster}]\label{unique-prop} A pre-canonical structure admits at most one canonical basis.
\end{proposition}

It is usually 
 difficult to describe   elements of a canonical basis explicitly. 
 However, one can often at least guarantee that a canonical basis exists.
Continue to assume $V$ is a free $\cA$-module with a pre-canonical structure $(\psi, \{a_c\})$ whose standard basis is indexed by $(C,\leq)$. 

\begin{theorem}[Du \cite{Du}] \label{canon-thm}
If all lower intervals $(-\infty,x]  = \{c \in C : c \leq x\}$ in  the partially ordered index set $(C,\leq)$ are finite
then the pre-canonical structure on $V$ admits a canonical basis.
\end{theorem}

\begin{proof}
The result is equivalent to \cite[Theorem 1.2 and Remark 1.2.1(1)]{Du}.
One can also adapt the argument   Lusztig gives in \cite[Section 4.9]{LV2},
 which proves the existence of a canonical basis in one particular pre-canonical structure but  makes sense in greater generality. 
\end{proof}
 
 Webster lists several examples of pre-canonical structures from representation theory in the introduction of \cite{Webster}. 
Pre-canonical structures, such as in these examples, arise naturally from graded categorifications, by which we broadly mean    isomorphisms
\be\label{iso} V \xrightarrow{\ \sim \ } [ \sC ]\ee
where $\sC$ is an additive category with $\ZZ$-graded objects, and $[\sC]$ denotes its split Grothendieck group: this is the $\cA$-module generated by the symbols $[C]$ for objects $C \in \sC$, subject to the relations $ [A] + [B] = [C]$ whenever $A \oplus B \cong C$
%
and 
$v^n[C] = [C(n)]$ where $C(n)$ is the object $C \in \sC$ with its grading shifted down by $n$.
The bar involution of a pre-canonical structure on $V$ should then correspond via \eqref{iso} to a duality functor on $\sC$, and elements of the standard basis should arise as some set of easily located objects in $\sC$, each of which contains a unique indecomposable summand not found in smaller objects.
A canonical basis in turn should correspond to a representative set of indecomposable objects which are self-dual with respect to some choice of grading shift.

\begin{example}
The pre-canonical structure on $V=\cA$ comes from the categorification taking $\sC$ to be the category of finitely generated $\ZZ$-graded free $R$-modules (with $R$ any commutative ring), with morphisms given by grading preserving $R$-linear maps.
For this category, there is a unique ring isomorphism $\cA \xrightarrow{\sim} [\sC]$ identifying $1 \in \cA$ with $[\one] \in [\sC]$, where $\one$ denotes the graded $R$-module whose $n$th component is $R$ when $n=0$ and is $0$ otherwise.
The bar involution $f\mapsto \overline f$ on $\cA$ is the decategorification 
of the duality functor 
$ M \mapsto \Hom(M,\one)$
where $\Hom(M,\one)$ denotes the graded $R$-module whose $n$th component is the set of grading preserving $R$-linear maps $M \to \one(n)$.
\end{example}

In general, confronted with some natural pre-canonical structure, it is an interesting problem (which  in the present work we do not address) to identify a categorification which can explain the existence and special properties of an associated canonical basis. 

\subsection{Comparison with IC bases and $P$-kernels}
 \label{IC-sect}

Webster's definition of canonical bases  is similar to two  concepts  appearing earlier in the literature:  \emph{IC bases} as formalized by Du in \cite{Du} and \emph{$P$-kernels} as introduced by Stanley in \cite{StanleyPKernel}. 
We review this terminology here, and explain how one may view canonical bases as special cases of IC bases, and $P$-kernels as special cases of pre-canonical structures. 
We remind the reader that for us, all pre-canonical structures (as specified by Definition \ref{precanon-def}) are what Webster \cite[Definition 1.5]{Webster} calls \emph{balanced pre-canonical structures.}

To begin, we recall the following definition from \cite{Du},   studied elsewhere, for example, in \cite{Brenti1,Green}.

\begin{definition}
Let $V$ be a free $\cA$-module with 
\begin{itemize}
\item a ``bar involution''   $\psi $ given by an $\cA$-antilinear map $V \to V$
with $\psi^2=1$. 
\item a ``standard basis'' $\{a_c\}$ with index set $C$. 
\end{itemize}
A set of vectors $\{b_c\}$ of $V$ is an \emph{IC basis} relative to $(\psi,\{a_c\})$ if it is the \emph{unique} basis such that 
\[ \psi(b_c) = b_c \qquand b_c \in a_c + \sum_{c' \in C} v^{-1}\ZZ[v^{-1}] \cdot a_{c'}\qquad\text{for each $c \in C$.}\]
\end{definition}

\begin{remark}
In \cite{Brenti1,Du,Green}, this definition is formulated slightly differently. There, one begins with a bar involution $\psi$, a basis $\{m_c\}_{c \in C}$ of $V$, and a function $r : C \to \ZZ$. An IC basis of $V$ is then defined exactly as above relative to $\psi$ and the  standard basis $\{a_c\}$ 
given by setting $a_c = v^{-r(c)} m_c$. One passes to our definition by assuming $r=0$; there is clearly no loss of generality in this reduction.
\end{remark}

The initial data in the definition of an IC basis is more general than a pre-canonical structure in two aspects: there is no condition on the action of the bar involution on the standard basis, and the index set $C$ is no longer required to be partially ordered.
When the initial data  $(\psi,\{a_c\})$ is a pre-canonical structure, the notions of a canonical basis and an IC basis are equivalent:

\begin{proposition}\label{IC-prop} Let $V$ be a free $\cA$-module with a pre-canonical structure $(\psi,\{a_c\})$.
Relative to $(\psi,\{a_c\})$, a set of vectors $\{b_c\}$ in $V$ is a canonical basis  if and only if it is an IC basis.
\end{proposition}

\begin{proof}
Suppose $\{b_c\}$ is an IC basis relative to $(\psi,\{a_c\})$. Let $f_{x,y} \in v^{-1}\ZZ[v^{-1}]$ for $x,y \in C$ be the polynomials such that  $b_y = a_y + \sum_{x \in C} f_{x,y} a_{x}$.
To show that $\{b_c\}$ is a canonical basis, we must check that $f_{x,y} = 0$ whenever $x \not < y$. This follows since if $y \in C$ is fixed and $x \in C$ is maximal among all elements $x\not < y$ with $f_{x,y} \neq 0$, then 
the equality $b_y = \psi(b_y)$ together with the unitriangular formula for $\psi$ implies that $f_{x,y} =\overline{f_{x,y}}$, which is impossible for a nonzero element of $v^{-1}\ZZ[v^{-1}]$. 

Now suppose conversely that $\{b_c\}$ is a canonical basis. This basis automatically has both desired properties of an IC basis, so it remains only to show that it is the unique basis with these properties. This follows from Proposition \ref{unique-prop}, since the argument in the previous paragraph shows that any other basis $\{b'_c\}$ 
with the desired properties of an IC basis is a canonical basis.
%
%
\end{proof}

\def\Int{\mathrm{Int}}

Stanley first introduced in   \cite{StanleyPKernel} the concept of a $P$-kernel for any locally finite poset $P$, which Brenti studied subseqently in \cite{Brenti0,Brenti1}.
 To define $P$-kernels we must review some terminology for partially ordered sets; \cite[Chapter 3]{StanleyEnum1} serves as the standard reference for this material.

Let $P$ be a partially ordered set (i.e., a poset) and let $\Int(P) = \{ (x,y) \in P^2 : x \leq y \}$.
Assume the poset $P$ is \emph{locally finite}, i.e., that  $\{ t \in P : x\leq t \leq y\}$ is finite for all $x,y \in P$.  
Let $R$ be a commutative ring and let $q$ be an indeterminate. The \emph{incidence algebra} $I(P;R[q])$ is the set of functions $f : \Int(P) \to R[q],$ with sums and scalar multiplication given pointwise and 
products given by 
\[ (fg)(x,y) = \sum_{x \leq t \leq y} f(x,t)g(t,y)\qquad\text{for }f,g : \Int(P) \to R[q].\]
This algebra has a unit given by the function $\delta_P : \Int(P) \to R[q]$ with $\delta_P(x,y) = \delta_{x,y}$ for $x,y \in P$. 
A function $f : \Int(P) \to R[q]$ is invertible if and only if $f(x,x)$ is a unit in $R[q]$ for all $x \in P$.
We adopt the  convention of setting $f(x,y) = 0$ whenever $f : \Int(P) \to R[q]$ and $x,y \in P$ are elements such that $x \not \leq y$.

Finally, let $r : P \to \ZZ$ be a function such that $r(x) <r(y)$ if $x<y$, and define $r(x,y) = r(y)-r(x)$ for $x,y \in P$. Relative to the initial data $(P,R,q,r)$, we have the following definition, which can be found as \cite[Definition 6.2]{StanleyPKernel} or in \cite[Section 2]{Brenti1}.

\begin{definition}
 An element $K \in I(P;R[q])$ is a \emph{$P$-kernel} if 
 \ben
 \item[(1)] 
 $K(x,x) = 1$ for all $x \in P$.
 \item[(2)] There exists an invertible $f \in I(P;R[q])$ such that $(Kf)(x,y) = q^{r(x,y)} \overline{f(x,y)}$ for   $x,y \in P$.
 \een
  An invertible element $f \in I(P)$ satisfying condition (2) is called \emph{$K$-totally acceptable}.
  \end{definition}


Brenti proves the following result as \cite[Theorem 6.2]{Brenti0}. This statement strengthens  an earlier result \cite[Corollary 6.7]{StanleyPKernel} due to Stanley.
  
  \begin{theorem}[Brenti \cite{Brenti0}] \label{Pkernel-thm}
Suppose $K \in I(P;R[q])$ 
is a $P$-kernel. If  $P$ is locally finite, then there exists a unique $K$-totally acceptable element $\gamma \in I(P;R[q])$ such that 
\[ \gamma(x,x) = 1 \qquand \deg_q (\gamma(x,y)) < \tfrac{1}{2}r(x,y) \qquad\text{for all }x,y \in P\text{ with }x<y.\]
Call  $\gamma$  the \emph{KLS-function} of $K$. (Here ``KLS'' abbreviates ``Kazhdan-Lusztig-Stanley.'')
\end{theorem}


Returning to our earlier convention,
 we let $C$ be an index set with a partial order $\leq$. Assume the hypothesis of Theorem \ref{canon-thm} (i.e., that all lower intervals in $C$ are finite) and let $V$ be the free $\cA$-module with a basis given by the symbols $a_c$ for $c \in C$.
To translate the language of $P$-kernels into pre-canonical structures, 
assume $P = (C,\leq)$ and $R = \ZZ$ and $q=v^2$. 
Given a $P$-kernel $K$, we may then define 
$\psi_K : V \to V$ as the $\cA$-antilinear map with \[\psi_K(a_y) =  \sum_{x \in C} v^{r(x,y)}  \cdot \overline{K(x,y)} \cdot a_{x}\qquad\text{for $y \in C$.}\]
Note that our assumption that $C$ has finite lower intervals ensures that the sum on the right side of this formula is well-defined.

In \cite{Brenti0}, Brenti proves that $P$-kernels are equivalent to IC bases of a special form.
It turns out that this special form is essentially the requirement that the initial data $(\psi,\{a_c\})$ of an IC basis form a pre-canonical structure. Brenti's results thus translate via Proposition \ref{IC-prop} into the following statement relating $P$-kernels and canonical bases.

\begin{theorem}[Brenti \cite{Brenti0}]
\label{Pkernel-prop}
Assume $P = (C,\leq)$ and $R = \ZZ$ and $q=v^2$.
\ben
\item[(a)] The map $K \mapsto \psi_K$  is a bijection from the set of $P$-kernels to the set of maps $\psi$ such that $(\psi,\{a_c\})$ is a pre-canonical structure  on $V$ with the property that
\[\psi(a_y) \in \ZZ[v^{-2}]\spanning\{ v^{r(x,y)}a_{x} : x \in C\}\qquad\text{for $y \in C$.}\]
\item[(b)] If $\gamma$ is the KLS-function of a $P$-kernel $K$ and $\{b_c\}$ is the canonical basis of $V$ relative to the pre-canonical structure $(\psi_K,\{a_c\})$, 
then 
\[ b_y = \sum_{x \in C} v^{-r(x,y)} \cdot  {\gamma(x,y)} \cdot a_x\qquad\text{for $y \in C$.}\]
\een
\end{theorem}

\begin{remark}
Note that part (b)
is only a meaningful statement if the KLS-function $\gamma$ and the canonical basis $\{b_c\}$ both exist, but this follows from Theorems \ref{canon-thm} and \ref{Pkernel-thm}
since we assume all lower intervals in $P= (C,\leq)$ are finite.
Observe that since $\gamma(x,y) \in \ZZ[v^2]$ for all $x,y \in C$, this result shows that  not all canonical bases   correspond to KLS-functions of $P$-kernels.
\end{remark}

\begin{proof}
The definition of $\psi_K$ makes sense for any $K \in I(P,R[q])$, and 
part (a) is equivalent to
the statement that $(\psi_K,\{a_c\})$ is a pre-canonical structure if and only if $K$ is a $P$-kernel.
Clearly $(\psi_K,\{a_c\})$ is a pre-canonical structure if and only if $K(x,x) = 1$ for all $x \in P$ and $\psi_K^2=1$.
The assertion that these two properties hold if and only if $K$ is a $P$-kernel
is precisely \cite[Proposition 3.1]{Brenti0}, since the map $\iota$ defined in part (ii) of that result  is just  $\psi_{K^{-1}}$  (with $m_c = v^{r(c)} a_c$).
Part (b) is equivalent to \cite[Theorem 3.2]{Brenti0} by Proposition \ref{IC-prop}.
\end{proof}

\subsection{Pre-canonical module structures}

In this short section we introduce a useful variant of Definition \ref{precanon-def}. 
Suppose $\cB$ is an $\cA$-algebra with a pre-canonical structure; write $\overline b$ for the image of $b \in \cB$ under the corresponding bar involution. (For us, all algebras  are unital and associative.)
Let $V$ be a $\cB$-module which is free as an $\cA$-module.
\begin{definition}
A \emph{pre-canonical $\cB$-module structure} on $V$ is  a pre-canonical structure whose bar involution $\psi : V\to V$ commutes with the bar involution of $\cB$ in the sense that 
\[ \psi( bx) = \overline b \cdot \psi(x)\qquad\text{for all }b \in \cB\text{ and }x \in V.\]
\end{definition}

Observe that a pre-canonical structure is thus the same thing as a pre-canonical $\cA$-module structure.
The additional compatibility condition satisfied by a pre-canonical $\cB$-module structure can be useful for proving uniqueness statements. In particular, we have the following lemma.

\begin{lemma}\label{uniquepre-lem} Suppose $V$ has a basis $\{a_c\}$ with partially ordered index set $(C,\leq)$. If $V$ is generated as a $\cB$-module by the minimal elements of the basis $\{a_c\}$, then there exists at most one pre-canonical 
$\cB$-module structure on $V$ in which $\{a_c\}$ serves as the ``standard basis.''
\end{lemma}

\begin{proof}
Suppose $\psi$ and $\psi'$ are two $\cA$-antilinear maps $V\to V$ which, together with $\{a_c\}$, give $V$ a pre-canonical $\cB$-module structure. Let $U \subset V$ be the set of elements on which $\psi$ and $\psi'$ agree. Then $U$ is a $\cB$-submodule which contains 
the minimal elements of the basis $ \{a_c\}$. Since these elements generate $V$, we have $U = V$ so $\psi = \psi'$.
\end{proof}

%


%
%
%

\subsection{Twisted involutions}\label{twisted-sect}
 
 We review here the definition of the set of twisted involutions attached to a Coxeter system. This set has many   interesting combinatorial properties; see \cite{H1,H2,H3,Incitti0,Incitti,S}.
Let $(W,S)$ be any Coxeter system. Write $\ell : W \to \NN$ for the associated length function and $\leq $ for the Bruhat order. 
We denote by  $\Aut(W,S)$  the group of automorphisms $\theta:W \to W$ such that $\theta(S) = S$,
and define 
\[W^+= \{ (x,\theta): x \in W\text{ and }\theta \in \Aut(W,S)\}.\] 
We extend the length function and Bruhat order to $W^+$ by setting $\ell(x,\theta) = \ell(x)$
and by setting $(x,\theta) \leq (x',\theta')$ if and only if $\theta = \theta'$ and $x\leq x'$.
The set $W^+$ has the structure of a group, in which multiplication of elements is given by 
\[ (x,\alpha)(y,\beta) = (x\cdot \alpha(y),\alpha\beta).\]
We view $W \subset W^+$ as a subgroup by identifying $x \in W$ with the pair $(x,1)$. Likewise, we view $\Aut(W,S)\subset W^+$ as a subgroup by identifying $\theta \in \Aut(W,S)$ with the pair $(1,\theta)$.
With respect to these inclusions, $W^+$ is a semidirect product $  W \rtimes \Aut(W,S)$.
\begin{definition}
The set of \emph{twisted involutions} of a Coxeter system $(W,S)$ is 
 \[ \I = \I(W,S) = \{ w \in W^+ : w=w^{-1}\}.\]
\end{definition}
A pair $(x,\theta) \in W^+$ belongs to $\I$ if and only if $\theta = \theta^{-1}$ and 
 $\theta(x) = x^{-1}$. In this situation, often in the literature the element $x \in W$ is referred to as a twisted involution, relative to the automorphism $\theta$.
We have defined twisted involutions slightly more generally as ordinary involutions of the extended group $W^+$, since 
all of the results we will state are true relative to any choice of automorphism $\theta$.

 If $s \in S$ and $w  = (x,\theta)\in \I$ then $sws = (s\cdot x\cdot \theta(s),\theta)$ is also a twisted involution. The latter may be equal to $w$; in particular, $sws = w$  if and only if $sw=ws$, in which case $sw \in \I$. 
 \begin{notation}  Let $s\act w$ denote whichever of $sws$ or $sw$ is in $\I\setminus\{w\}$; i.e.,  define
\be\label{sact-eq} s\act w  = \begin{cases} sws&\text{if }sw\neq ws \\ sw&\text{if }sw=ws\end{cases}
\qquad\text{for }s \in S\text{ and }w \in \I.
\ee 
While $s\act (s\act w) = w$,
the operation $\act$ does not extend to an action of $W$ of $\I$. 
\end{notation}

The restriction of the Bruhat order on $W$ to $\I$ 
forms a poset with many special properties. Concerning this, 
we will just need the following result, 
which rephrases  \cite[Theorem 4.8]{H1}.

\begin{theorem}[Hultman \cite{H1}]  \label{rho-def} 
The poset $(\I,\leq)$ is graded, and its rank function $\rho : \I \to \NN$ 
satisfies 
\[
\rho (s\act w) = \rho(w) - 1
\quad\Leftrightarrow\quad \ell(s\act w) < \ell(w)
\quad\Leftrightarrow\quad \ell(sw) = \ell(w)-1
\]
for all $s \in S$ and $w \in \I$.
\end{theorem}

We reserve the notation $\rho$ in all later sections to denote the rank function of $(\I,\leq)$. Note that $\rho(1) = 0$, and so one can compute $\rho(w)$ inductively using the equivalent identities in the theorem.
As with $\ell$, there are explicit formulas for $\rho$   when $W$ is a classical Weyl group; see \cite{Incitti0,Incitti}.

\subsection{Kazhdan-Lusztig basis}\label{KL-sect}

In this final preliminary section we recall briefly  the definition of the Kazhdan-Lusztig basis of the Iwahori-Hecke algebra of a Coxeter system. 
As references for this material, we mention \cite{CCG,KL,Soergel}.
Continue to let $(W,S)$ be a Coxeter system with length function $\ell : W \to \NN$ and Bruhat order $\leq$.
We write $\H = \H(W,S)$ to denote the free $\cA$-module with a basis given the symbols $H_w$ for $w \in W$. There is a unique $\cA$-algebra structure on $\H$ such that 
\[ H_s H_w = \begin{cases} H_{sw}&\text{if }sw>w \\ H_{sw} + (v-v^{-1})H_w &\text{if }sw<w\end{cases}
\qquad\text{for }s \in S\text{ and }w \in W.\]
The \emph{Iwahori-Hecke algebra} of $(W,S)$ is $\H$ equipped with this structure. 
%
%
%

The unit of $\H$ is the basis element $H_1$, which often we  write as $1$ or simply omit.
Observe that $H_s^{-1} = H_s + (v^{-1}-v)$ and that $H_w = H_{s_1}\cdots H_{s_k}$ whenever $w=s_1\cdots s_k$ is a reduced expression. 
Hence  every basis element $H_w$ for $w \in W$ is invertible.
We  denote by $H \mapsto \overline H$ the $\cA$-antilinear map $\H \to\H$ with $\overline { H_w} = (H_{w^{-1}})^{-1}$ for  $w \in W$. One checks that this map is  a ring involution, and we have the following result from Kazhdan and Lusztig's seminal work \cite{KL}.

\begin{thmdef}[Kazhdan and Lusztig \cite{KL}] \label{kl-thm}
Define
\begin{itemize}
\item  the ``bar involution'' of $\H$ to be the map $H\mapsto \overline H$.
\item  the ``standard basis'' of $\H$ to be $\{H_w\}$ with the partially ordered index set $(W,\leq)$.
\end{itemize}
This is a pre-canonical structure on $\H$ and it admits a canonical basis $\{ \underline H_w\}$.
\end{thmdef}



The canonical basis $\{ \underline H_w\}$ is the \emph{Kazhdan-Lusztig basis} of $\H$.
It is a simple exercise to show  for $s \in S$ that 
$\underline H_s   = H_s + v^{-1}.$
Define $h_{y,w} \in \ZZ[v^{-1}]$ for $y,w \in W$  such that 
  $\underline H_w  =  \sum_{y \in W} h_{y,w} H_y$.  We note the following well-known property of these polynomials. 

\begin{proposition}[Kazhdan and Lusztig \cite{KL}]\label{degbound1-prop} If $y\leq w $  then $v^{\ell(w) - \ell(y)}h_{y,w} \in 1 + v^2 \ZZ[v^2].$
\end{proposition}

\begin{remark} Define $q=v^2$ and $P_{y,w} = v^{\ell(w) - \ell(y)}h_{y,w}$ for $y,w \in W$. The polynomials $P_{y,w} \in \ZZ[q]$ are usually called the \emph{Kazhdan-Lusztig polynomials} of the Coxeter system $(W,S)$.
\end{remark}

The Kazhdan-Lusztig basis has several remarkable positivity properties;
for example, it is now known from work of Elias and Williamson \cite{EW} that 
for all $x,y \in W$ one has
$h_{x,y} \in   \NN[v^{-1}] $ and $ \underline H_x \underline H_y \in \NN[v,v^{-1}]\spanning\{ \underline H_z : z \in W\}$.
Available proofs  of such phenomena make extensive use of the interpretation of the Iwahori-Hecke algebra $\H$  as the split Grothendieck of an appropriate category (in \cite{EW}, the category of Soergel bimodules).
 This is an important motivation for the problem of constructing categorifications which give rise to pre-canonical structures of interest.  


\section{Characterizations}

The results of Kazhdan and Lusztig in the previous section give us a canonical basis for the free $\cA$-module generated by any Coxeter group $W$. In turn, recent results of Lusztig and Vogan \cite{LV2,LV1,LV6} construct a canonical basis of the free $\cA$-module generated by the set of twisted involutions in $W$. 
In this section our goal, broadly speaking, is to characterize the ways one can modify such constructions to get other canonical bases, and to explain how such bases differ from each other.


\subsection{Morphisms for pre-canonical structures}
\label{morph-sect}

To this end, our first task is to describe what it means for two pre-canonical structures to be the same.
This amounts to defining   what should comprise a morphism between pre-canonical structures on free $\cA$-modules. In this pursuit we are guided by the principle that 
if a morphism exists from one pre-canonical structure to another, and if the first structure admits a canonical basis, then the second structure should  admit a canonical basis which can be  described explicitly in terms of the first basis.

The following is a natural but rigid notion of (iso)morphism compatible with this philosophy.
Suppose $V$ and $V'$ are free $\cA$-modules with respective pre-canonical structures $(\psi, \{a_c\})$ and $(\psi',\{a'_c\})$.
We say that an $\cA$-linear map $\varphi : V \to V'$ is a \emph{strong isomorphism} of pre-canonical structures if $\varphi$ restricts to an order-preserving bijection $\{a_c\} \to \{a'_c\}$ between standard bases and $\varphi$ commutes with  bar involutions
in the sense that $\varphi\circ \psi = \psi'\circ \varphi.$
  Under these conditions,  $\varphi$ is necessarily invertible as an $\cA$-linear map.
The inverse and composition of strong isomorphisms of pre-canonical structures are again  strong isomorphisms of pre-canonical structures.
Moreover, if $\varphi : V \to V'$ is a strong isomorphism of pre-canonical structures and $V$ admits a canonical basis $\{b_c\}$, then $\{ \varphi(b_c)\}$ is a canonical basis of $V'$.
%

There are other situations in which we would like to consider two pre-canonical structures to be ``the same'' besides when they are strongly isomorphic. We illustrate this as follows.
Continue to let $V$ be a free $\cA$-module with a pre-canonical structure $(\psi,\{a_c\})$ whose standard basis is indexed by $(C,\leq)$.
Suppose for each index $c \in C$ we have an element $d_c \in \cA$. Let $u_c = d_c  a_c$ and consider the set of rescaled basis elements $\{u_c\}$, likewise indexed by $(C,\leq)$.
These elements  are linearly independent if and only if each $d_c \neq 0$, so assume this condition holds and define $U = \cA\spanning\{ u_c : c \in C\}$. 
One naturally asks when $(\psi,\{u_c\})$ is a pre-canonical structure on the submodule $U \subset V$.
Since we have
\[ \psi(u_c) \in \tfrac{\overline{d_c}}{d_c} \cdot u_c + \sum_{c'<c} \cA \cdot \tfrac{\overline{d_c}}{d_{c'}}\cdot  u_{c'}\]
it follows that $(\psi,\{u_c\})$ is a pre-canonical structure on $U$ at least  when (i) each $d_c = \overline{d_c}$ and (ii) $d_{c}= q_{c',c}  d_{c'}$ for some $q_{c',c} \in \cA$ whenever $c' < c$ in $C$. Moreover,   the first of these sufficient conditions is also necessary. 
Note that if (i) and (ii) hold then $q_{c',c} = \overline{q_{c',c}}$ and so $q_{c',c} \in \ZZ[v+v^{-1}]$
since $\ZZ[v+v^{-1}]$ is the set of bar invariant elements of $\cA$.

Assume  conditions (i) and (ii) hold
 and further that $V$ admits a canonical basis $\{b_c\}$ with respect to the pre-canonical structure $(\psi,\{a_c\})$. If $U$ also has a canonical basis, then one   asks how it is related to the basis
  $\{a_c\}$; in particular, when does some rescaling of $\{b_c\}$ give a canonical basis for $U$?
By condition (C2) in Definition \ref{canon-def}, it follows that the only possible such basis would be given by $\{ d_c b_c\}$. 
Since
 \[ d_c b_c \in  u_c + \sum_{c'<c} v^{-1}\ZZ[v^{-1}]\cdot q_{c',c} \cdot u_{c'}\]
 it follows that $\{d_cb_c\}$ is a canonical basis for $U$ at least when $q_{c',c} \in \ZZ[v^{-1}]$.
 Since  $\ZZ  = \ZZ[v^{-1}] \cap \ZZ[v+v^{-1}]$, 
 we may summarize this discussion with the following lemma.
%
 
 \begin{lemma} \label{rescale-prop}
 For each index $c \in C$ let $d_c \in \cA$ and define
 \[u_c = d_c  a_c
\qquand U = \cA\spanning\{ u_c : c \in C\}. \] 
 Suppose the following conditions hold:
 \begin{itemize}
 \item[(i)] $d_c \in \ZZ[v+v^{-1}]$ and $d_c\neq 0$ for all $c \in C$.
 \item[(ii)] $d_c/ d_{c'} \in \ZZ $ whenever $c'<c$.
 \end{itemize}
 Then $(\psi,\{u_c\})$ is a pre-canonical structure on $U$. If $\{ b_c\}$ is a canonical basis of $V$ then $\{ d_c b_c\}$ is a canonical basis of $U$.
 \end{lemma} 
 
 Morphisms between pre-canonical structures should at least include strong isomorphisms and also the $\cA$-linear maps  $D : V \to V'$ 
 given by $D(a_c) = d_c  a_c$ when the conditions hold in the preceding proposition.
 There is a third kind of map which should form a morphism; in particular, it is natural to consider the map $\Phi $ given by \eqref{hh2} to be a morphism between the pre-canonical structures on $\H$ and $\H_2$, as we will see in the following lemma.
 
 Let $\epsilon$ be a ring endomorphism of $\cA$. Such a map is $\ZZ$-linear and  completely determined by its value at $v \in \cA$, which must be a unit, since $\epsilon(v) \epsilon(v^{-1}) = \epsilon(vv^{-1})  = \epsilon(1) = 1$. It follows that $\epsilon(v) = \pm v^n$ for  some $n \in \ZZ$. 
 Call $n$ the \emph{degree} of the endomorphism $\epsilon$.
 We say that a map $\varphi : M\to N$ between $\cA$-modules is \emph{$\epsilon$-linear} if $\varphi(fm) = \epsilon(f) \varphi(m)$ for $f \in \cA$ and $m \in M$.
 
 \begin{lemma}\label{twist-prop}
  Let $\epsilon$ be a ring endomorphism of $\cA$ and write $\tau: V \to V$ and $\phi : V\to V$
 for the respective $\epsilon$-linear
 and
 $\cA$-antilinear 
  maps with 
  \[\tau(a_c) = a_c
  \qquand \phi(a_c) = \tau\circ \psi(a_c)\qquad\text{for $c \in C$.}\]
Then
  $( \phi, \{ a_c\} )$ is another pre-canonical structure on $V$.
If $\{ b_c\}$ is a canonical basis of $V$ relative to $\{\psi,\{a_c\})$ and $\epsilon$ has positive degree, then   $\{ \tau(b_c)\}$ is a canonical basis of $V$ relative to  $(\phi,\{a_c\})$.

%
 \end{lemma}
 
 \begin{proof}
That  $( \phi, \{ a_c\} )$ is a pre-canonical structure is clear from the definitions, and
checking that $\{\tau(b_c)\}$ is a canonical basis relative to this structure
is straightforward.
 \end{proof}

Motivated by the preceding lemmas, we adopt the following   definition.
Let $V$ and $V'$ be free $\cA$-modules with pre-canonical structures $(\psi, \{a_c\})$ and $(\psi',\{a'_c\})$. 
Assume the standard bases $\{a_c\}$ and $\{a'_c\}$ have the same partially ordered index set $(C,\leq)$.

 \begin{definition}\label{morphism-def}
A  map $\varphi : V \to V'$  is a \emph{morphism} of pre-canonical structures if
\ben

\item[(i)] The map $\varphi$ is $\epsilon$-linear for a positive degree ring endomorphism $\epsilon: \cA \to \cA$. 

\item[(ii)] There
are nonzero polynomials $d_c \in \cA$ for  $c \in C$  with $d_c/d_{c'} \in \ZZ$ whenever $c'<c$, such that if
 $D : V \to V$ is the  $\cA$-linear map
with
$D(a_c) = d_c a_c$ for   $c \in C$
then
$ \psi' \circ \varphi = \varphi\circ  \psi ^ D
$, 
where we define
$\psi^D =  D^{-1}\circ  \psi \circ D.$
\een
 \end{definition}
 
 \begin{remark}
The polynomials $d_c$ in condition (ii) automatically belong to $\ZZ[v+v^{-1}]$ 
since the coefficients of $a_c$ in $\varphi^{-1}\circ \psi' \circ \varphi(a_c)$ and in $\psi^D(a_c)$, which must be equal, are 1 and $\overline{d_c} /  {d_c}$ respectively.
This observation and the fact that $d_c/d_{c'} \in \ZZ$ whenever $c'<c$ in $C$ ensure that  $\psi^D$ is a well-defined map $V \to V$, even though $D^{-1}$   may not be.
 \end{remark}
 
 If $\varphi : V \to V'$ is a morphism of pre-canonical structures then we call a map $D : V \to V$ of the form in condition (ii) of Definition \ref{morphism-def} a \emph{scaling factor} of $\varphi$. 
 If $V' \subset V$ and $\varphi$ is equal to one of its scaling factors then we call $\varphi$ a \emph{scaling morphism}.
 We define the \emph{degree} of any morphism $\varphi$  to be the degree of 
 the ring endomorphism $\epsilon$ in condition (i). If $V=V'$ and $\{a_c\} = \{a'_c\}$ and the identity is a scaling factor of $\varphi$, then we call $\varphi$ a \emph{parametric morphism}.
 
 In the rest of this section we describe some properties of morphisms in this sense.
 We fix some notation.
 Let $V$ and $V'$ and $V''$ be free $\cA$-modules with pre-canonical structures $(\{a_c\},\psi)$ and $(\{a'_c\},\psi')$ and $(\{a''_c\},\psi'')$. Assume
 the standard bases of these structures all have the same partially ordered index set $(C,\leq)$,
and suppose $\varphi : V \to V'$ and $\varphi' : V' \to V''$ are morphisms of pre-canonical structures.

 \begin{proposition}
The composition $ V\xrightarrow{\varphi} V' \xrightarrow{\varphi'}  V''$ is  a morphism of pre-canonical structures.
The collection of pre-canonical structures on free $\cA$-modules forms a category. 
\end{proposition}

 
The proposition follows in an elementary way from the definitions; we omit its proof.
 
 \begin{proposition}
Every morphism of pre-canonical structures is equal to some composition $  \iota \circ \sigma\circ \tau $ where $\iota$ is a strong isomorphism, $\sigma$ is a scaling morphism, and $\tau$ is a parametric morphism.
 \end{proposition}
 
 \begin{proof}
  Let $\epsilon$ be the $\cA$-endomorphism of positive degree such that $\varphi$ is $\epsilon$-linear.
Define $\tau  : V \to V$ and $\phi : V \to V$, relative to $(\psi, \{a_c\})$ and $\epsilon$, as in Lemma \ref{twist-prop}. Then $(\psi,\{a_c\})$ and $(\phi,\{a_c\})$ 
are both pre-canonical structures on $V$ and 
$ \tau : V \to V$ is a parametric morphism from the first to the second.

Next, let $D$ be a scaling factor of $\varphi$ so that $D(a_c) = d_c  a_c$
 for
 some
  $d_c \in \ZZ[v+v^{-1}]$ for each $c \in C$. Let $d'_c = \epsilon(d_c)$ and write $\sigma : V \to V$ for the $\cA$-linear map with $\sigma(a_c) =d'_c  a_c$. 
Define $u_c = d'_c  a_c$ and  $U = \cA\spanning\{ u_c : c \in C\}$ as in Lemma \ref{rescale-prop}. Then $(\phi,\{u_c\})$ is a pre-canonical structure on $U$ and the map 
$\sigma : V \to U$ is a scaling morphism from $(\phi,\{a_c\})$ to $(\phi,\{u_c\})$.  

Finally, define $\iota : U \to V'$ as the $\cA$-linear map with $\iota(u_c) = a'_c$
for $c \in C$. This is a strong isomorphism since for any $c \in C$ we have  
\[ \iota \circ \phi(u_c) = d'_c\cdot \iota \circ \tau \circ \psi(a_c)  = \varphi \circ \psi^D(a_c) = \psi' \circ \varphi(a_c) = \psi'(a'_c) = \psi' \circ \iota(u_c).\]
As both $\iota \circ \psi$ and $\psi' \circ \iota$ are   $\cA$-antilinear, this identity shows that the two maps are equal. 
The composition $\iota \circ \sigma\circ \tau$ agrees with $\varphi$ at each basis element $a_c$, and both maps are $\epsilon$-linear, so they are equal.
\end{proof}
 
 \begin{proposition}\label{canon-prop}
 Suppose the pre-canonical structure on $V$ admits a canonical basis $\{b_c\}$. Then the pre-canonical structure on $V'$ also admits a canonical basis $\{ b'_c\}$.
If  $D$ is a scaling factor of $\varphi$ and  $\beta : V \to V$ is the $\cA$-linear map with $\beta(a_c) = b_c$ for each $c \in C$,
  then the composition 
  \[\varphi \circ D^{-1} \circ \beta \circ D \circ \beta^{-1}\]
  is a well-defined map $V \to V'$ which restricts to
   an order-preserving bijection $\{ b_c\} \to\{b'_c\}$.
 \end{proposition}
 
 \begin{proof}
 Let $b'_c = \varphi\circ D^{-1}\circ \beta \circ D \circ \beta^{-1}( b_c)
 $. It suffices to check that this element satisfies the defining conditions of a canonical basis. 
 This is a simple exercise which is left to the reader.
\end{proof}
 
 \begin{proposition}
 A morphism of pre-canonical structures is an isomorphism 
 (that is, there exists a morphism of pre-canonical structures which is its left and right inverse)
  if and only if it has  degree 1 and it has a scaling factor whose eigenvalues are each $\pm 1$.
 \end{proposition}
 
 \begin{proof}
If $\varphi$ has degree 1 and a scaling factor $D$ whose eigenvalues are each $\pm 1$, then $D=D^{-1}$ and $\varphi$ is an $\epsilon$-linear bijection (where $\epsilon = \epsilon^{-1}$ is a ring involution of $\cA$) and it follows that the inverse map $\varphi^{-1}$ is well-defined and a morphism of pre-canonical structures with scaling factor 
$\varphi \circ D \circ \varphi^{-1}$. 
Hence in this case $\varphi$ is an isomorphism of pre-canonical structures. 
Suppose conversely that $D$ is a scaling factor for $\varphi$ and that $\varphi^{-1}$ exists and is a morphism with scaling factor $D'$. Then $\varphi$ must have degree 1 since otherwise $\varphi$ is not invertible. To show that $\varphi$ has some scaling factor all of whose eigenvalues are $\pm 1$, let $D'' = \varphi\circ D \circ \varphi^{-1}$.
Then
\[ \psi' = \varphi\circ (\varphi^{-1} \circ \psi' \circ \varphi) \circ \varphi^{-1} = {D''}^{-1} \circ (\varphi \circ \psi \circ\varphi^{-1}) \circ D''
= 
(D'D'')^{-1} \circ \psi' \circ (D'D'').\]
For each $c \in C$ let $d_c$ and $d_c'$ be the elements of $\ZZ[v+v^{-1}]$ such that $D(a_c) = d_c a_c$ and $D'(a'_c) = d'_c  a'_c$.
Now, 
write $\sim$ for the minimal equivalence relation on $C$ such that $c \sim c'$ whenever $c,c' \in C$ such that the coefficient $f_{c',c}$ of $a'_{c'}$ in $\psi'(a'_c)$ is nonzero.  
The equation above implies 
\[f_{c',c} = {d_c}/{d_{c'}} \cdot {d'_c}/{d'_{c'}} \cdot f_{c',c}\]
so since $d_c/d_{c'}$ and $d'_c/d'_{c'}$ are both integers, these quotients must each be $\pm 1$.
Hence if $K$ is an equivalence class under $\sim$ then $d_c /d_{c'} \in \{\pm 1\}$ for any $c,c' \in K$. 
For each such equivalence class $K$, choose an arbitrary $c \in K$ and let $d_K = d_c$. 
Now let $E : V \to V$ be the $\cA$-linear map with $E(a_c) = d_K  a_c$ where $K$ is the equivalence class of $c \in C$. We   claim that 
\[ \psi = E^{-1} \circ \psi \circ E.\]
This follows since if the coefficient of $a_{c'}$ in $\psi(a_c)$ is some polynomial $f \in \cA$, then the 
coefficient of $a_{c'}$ in $E^{-1} \circ \psi \circ E(a_c) $ is $d_K/d_{K'}\cdot f$ where $K$ and $K'$ are the equivalence classes of $c$ and $c'$. If $f=0$ then these coefficients are both zero, and if $f\neq 0$
then the
coefficient of $a'_{c'}$ in $\psi'(a'_c)$ is also nonzero, so $K = K'$ and our coefficients are again equal.
From this claim, we conclude that $E^{-1}  D$ is another scaling factor of $\varphi$. The eigenvalues of this scaling factor are each $\pm 1$ since if $K$ is the equivalence class of $c \in C$ then $d_c / d_K \in \{ \pm 1\}$.
\end{proof}

The following corollary shows that  the structure constants of canonical bases arising from isomorphic pre-canonical structures
 differ only by a factor of $\pm 1$ or the substitution $v \mapsto -v$.

 \begin{corollary}\label{canon-cor} Suppose the pre-canonical structures on $V$ and $V'$ are isomorphic and 
 admit canonical bases $\{b_c\}$ and $\{b'_c\}$. Define $f_{x,y}(t), g_{x,y}(t) \in \ZZ[t]$ such that 
\[ b_y = \sum_{x \leq y} f_{x,y}(v^{-1})  a_{x}
\qquand
 b'_y = \sum_{x \leq y}g_{x,y}(v^{-1})  a'_{x}.
 \]
 Then for each $x,y\in C$ there are $\varepsilon_i \in \{\pm 1\}$ such that  $f_{x,y}(t) = \varepsilon_1\cdot  g_{x,y}(\varepsilon_2    t)$.
 \end{corollary}

\begin{proof}
Let $\varphi : V \to V'$ be an isomorphism of pre-canonical structures. By the previous proposition, $\varphi$ has a scaling factor $D$ whose eigenvalues are all $\pm 1$, and $\varphi$ is $\epsilon$-linear where $\epsilon \in \End(\cA)$ is either the identity or the ring homomorphism with $v \mapsto -v$. Given these considerations, the corollary  follows from Proposition \ref{canon-prop}.
\end{proof}

\subsection{Generic structures on group elements}
\label{4-sect}

In this and the  two sections which follow we consider  Hecke algebra modules of a certain generic form. We are interested in classifying such generic structures, saying which structures admit compatible pre-canonical structures, and identifying when such pre-canonical structures are isomorphic in the sense of Definition \ref{morphism-def}. The solutions to these problems will recover some constructions already studied in the literature, but will also reveal other structures not previously examined.
The unexpected existence of these ``extra'' solutions 
is the primary motivation for our results.

In this   section, the type of generic module structure which we study is a natural generalization of the regular representation of a Hecke algebra.
Our results here are useful 
mostly for comparison with the theorems in the next sections.
 The proofs in this section are only sketched, since they are just simpler versions of the arguments we use to 
 establish the results in Sections \ref{16-sect} and \ref{32-sect}.
 
%

\begin{notation}
If $X$ is a set then we
write $\cA X$ for the free $\cA$-module generated by $X$, and let $\End(\cA X)$ denote the $\cA$-module of $\cA$-linear maps $\cA X \to \cA X$.
A \emph{representation} of $\H$ in some $\cA$-module $\cM$ is an $\cA$-algebra homomorphism $\H \to \End(\cM)$. 
\end{notation}

Consider a $2\times 2$ matrix $\gamma = (\gamma_{ij})$ with entries in $\cA$. Given a Coxeter system $(W,S)$, we let $\rho_\gamma : \{ H_s : s \in S\} \to \End(\cA W)$ denote the map with
\[\label{form-eq0}
\rho_\gamma(H_s) ( w) =
\left\{
\ba 
\gamma_{11}\cdot {s w} \ +\ & \gamma_{12}\cdot w     				&&\quad\text{if $sw > w$} \\
\gamma_{21}\cdot {s w} \ +\ & \gamma_{22}\cdot w			&&\quad\text{if $sw< w$} 
\ea
\right.
\qquad\text{for $s \in S$ and $w \in W$}.
\]

\begin{definition}\label{HW-def}
The matrix $\gamma$ is an \emph{$(\H,W)$-structure}
if  for every Coxeter system $(W,S)$, 
the map $\rho_\gamma$ extends to a representation of $\H=\H(W,S)$ in $\cA W$.
\end{definition}

%
%
%

%
%

An $(\H,W)$-structure $\gamma = (\gamma_{ij})$ is \emph{trivial} if 
$\gamma_{11} = \gamma_{21} = 0$ and $\gamma_{12} = \gamma_{22} \in \{v,-v^{-1}\}$. Such a structure defines an $\H$-representation which decomposes as a direct sum of irreducible submodules given by free $\cA$-modules of rank one.
The definition of $\H$ affords an obvious example of a nontrivial $(\H,W)$-structure: namely, the matrix $\gamma$ with $\gamma_{11} = \gamma_{21} = 1$ and $\gamma_{12} = 0$ and $\gamma_{22} = v-v^{-1}$.

\begin{theorem}\label{HW-thm}
Every nontrivial  $(\H,W)$-structure is  equal to 
\[  \genrep{\alpha}{0}{\alpha^{-1}}{v-v^{-1}} \qquad\text{or}\qquad  \genrep{\alpha}{v-v^{-1}}{\alpha^{-1}}{0}\]
for some unit $\alpha$ in $\cA$.
All nontrivial  $(\H,W)$-structures define isomorphic $\H$-representations.

\end{theorem}

Recall that the units in the ring $\cA$ are the monomials of the form $\pm v^n$ for   $n \in \ZZ$.

\begin{proof}[Proof sketch]
The given matrices are $(\H,W)$-structures, since those on the left (respectively, right)   describe the action of $H_s$ for $s \in S$ on the basis $ \{ \alpha^{-\ell(w)} H_w : w \in W\}$ (respectively, $\{ \alpha^{-\ell(w)} \overline{H_w} :w \in W \}$) of $\H$. 
The corresponding $\H$-representations are evidently all isomorphic to the regular representation of $\H$ on itself.
That there are no other nontrivial $(\H,W)$-structures     follows by a simpler version of the argument used in the proof of Theorem \ref{HI-thm}. 
\end{proof}

An $(\H,W)$-structure  $\gamma$
defines an $\H$-module structure on $\cA W$ for every Coxeter system $(W,S)$. We say that $\gamma$ is \emph{pre-canonical} if each of these $\H$-modules has a pre-canonical $\H$-module structure   in which $W$ partially ordered by the Bruhat order is the ``standard basis.'' 
%
It follows from the preceding theorem and Lemma \ref{uniquepre-lem} that if $\gamma$ is nontrivial and pre-canonical, then there is a unique bar involution $\psi : \cA W \to \cA W$ such that $(\psi,W)$ is a pre-canonical $\H$-module structure. By Theorem \ref{canon-thm}, this pre-canonical structure always admits a canonical basis.


\begin{theorem}\label{4-thm}
Exactly 4 nontrivial $(\H,W)$-structures are pre-canonical. The 4 associated pre-canonical structures on $\cA W$ are all isomorphic (in the sense of Definition \ref{morphism-def}) to the pre-canonical structure on $\H$ given in Theorem-Definition \ref{kl-thm}. 
\end{theorem}

\begin{proof}[Proof sketch]
The proof is similar to that of  Theorem \ref{16-thm}. 
Let $\gamma$ be a nontrivial, pre-canonical $(\H,W)$-structure. Then $\gamma$ must be one of the two matrices in Theorem \ref{HW-thm} for some unit $\alpha \in \cA$. One first argues that $\alpha = \overline{\alpha}$ so that $\alpha \in \{ \pm 1\}$. Next, one observes that $\gamma$ remains pre-canonical if  $\alpha$ is replaced with $-\alpha$,
and that the  pre-canonical structures associated to these two $(\H,W)$-structures are always isomorphic. One may therefore assume $\alpha = 1$. It remains to prove that if $\gamma$ is the right-hand matrix in Theorem \ref{HW-thm} then its associated pre-canonical structure is isomorphic to the pre-canonical structure on $\H$ given  in Theorem-Definition \ref{kl-thm}.
This can be deduced from \cite[Lemma 2.1(i)]{KL}, after noting that the $\cA$-linear map with $w \mapsto \overline{H_w}$ 
 defines an isomorphism between $\cA W$ viewed as an $\H$-module via $\gamma$ and $\H$  viewed as a left module over itself.
\end{proof}

\subsection{Generic structures  on twisted involutions}
\label{16-sect}

In this section we introduce a second kind of generic  structure, which concerns Hecke algebra modules  on the space of twisted involutions in a Coxeter group.
Quite nontrivial results of Lusztig and Vogan \cite{LV2,LV1,LV6} provide interesting examples of this type of generic structure, which is what motivates their study.
Our results here depend on Lusztig and Vogan's work, which we review in Section \ref{exist-sect}. For this reason, we defer most proofs  to Section \ref{uniqueproof-sect}.

Consider a $4\times 2$ matrix $\gamma = (\gamma_{ij})$ with entries in $\cA$. Given a Coxeter system $(W,S)$, writing $\I = \I(W,S)$, we let $\rho_\gamma : \{ H_s : s \in S\} \to \End(\cA \I)$ denote the map with
\[\label{form-eq1} \rho_\gamma(H_s) (w) =
\left\{
\ba 
\gamma_{11}\cdot {s  ws} \ +\ & \gamma_{12}\cdot w     				&&\quad\text{if $s\act w = sws > w$} \\
\gamma_{21}\cdot {s  ws} \ +\ &  \gamma_{22}\cdot w			&&\quad\text{if $s\act w = sws< w$} \\
\gamma_{31}\cdot {s w} \ +\ &  \gamma_{32}\cdot w			&&\quad\text{if $s\act w =sw > w$} \\
\gamma_{41}\cdot {s w} \ +\ & \gamma_{42} \cdot w 	&&\quad\text{if $s\act w =sw < w$}
\ea
\right.
\qquad\text{for $s \in S$ and $w \in \I$}.
\]

\begin{definition}\label{HI-def}
The matrix $\gamma$ is an \emph{$(\H,\I)$-structure}
if for every Coxeter system $(W,S)$, the map
$\rho_\gamma$ extends to a representation of $\H=\H(W,S)$ in $\cA\I = \cA \I(W,S)$.
 \end{definition}
 
 It would make sense to view $\rho_\gamma$ as a map $ \{ H_s : s \in S\} \to \End(\cA W)$ by the same formula. However, combining some computations with the analysis in Section \ref{16proofs-sect}, one can show that $\rho_\gamma$ only extends to a representation of $\H$ in $\cA W$ for every Coxeter system $(W,S)$ when $\gamma$ is \emph{trivial}, where we say that $\gamma = (\gamma_{ij})$ is trivial if 
$\gamma_{11} = \gamma_{21} =\gamma_{31} = \gamma_{41} = 0$ and $\gamma_{12} = \gamma_{22}=\gamma_{32} = \gamma_{42} \in \{v,-v^{-1}\}$. 
Before we can classify the nontrivial $(\H,\I)$-structures, we need to describe the following basic notation of equivalence between  structures:

\begin{lemma}\label{diag-lem} Let $A,B,C,D,E,F,G,H \in \cA$ and suppose $\alpha,\beta \in \QQ(v)-\{0\}$ 
such that $A\alpha^{-1} $ and $C\alpha$ and $E\beta^{-1}$ and $G\beta$ all belong to $\cA$.
Let 
\[
\gamma = \menrep{A}{B}{C}{D}{E}{F}{G}{H}
\qquand
\gamma[\alpha,\beta]=  \menrep{A\alpha^{-1}}{B}{C\alpha}{D}{E\beta^{-1}}{F}{G\beta}{H}
.\]
If $\gamma$ is a $(\H,\I)$-structure then so is $\gamma[\alpha,\beta]$.
In this case, we say that $\gamma$ and $\gamma[\alpha,\beta]$ are \emph{diagonally equivalent}.
If $\alpha,\beta \in \cA$ then $\gamma$ and $\gamma[\alpha,\beta]$ define isomorphic representations of $\H$.

\end{lemma}

\begin{proof}
Assume $\gamma$ is an $(\H,\I)$-structure.
The $\H$-representation $\rho_\gamma$ extends to a representation in the larger $\cA$-module $\QQ(v)\I$ 
by linearity. Define $T : \QQ(v) \I \to \QQ(v)\I$ as the $\QQ(v)$-linear map with \[T(w) = \alpha^{\ell(w) - \rho(w)}\cdot  \beta^{2\rho(w)-\ell(w)}\cdot w\qquad\text{for $w \in \I$}\]
where
on the right
$\rho : \I \to \NN$ is defined as in Theorem \ref{rho-def}.
 Then $\gamma[\alpha,\beta]$ is an $(\H,\I)$-structure since 
$\rho_{\gamma[\alpha,\beta]}(H) =T^{-1} \circ \rho_{\gamma}(H) \circ T$ for all $H \in \H$.
\end{proof}

Let $u=v-v^{-1}$ and define four   $4\times 2$ matrices as follows:
\[ \Gamma= \menrep{1}{0}{1}{u}{1}{1}{u}{u-1} \qquad  \Gamma'= \menrep{1}{u}{1}{0}{1}{u-1}{u}{1}
\qquad 
 \Gamma''=\menrep{1}{0}{1}{u}{1}{-1}{-u}{u+1} \qquad \Gamma'''= \menrep{1}{u}{1}{0}{1}{u+1}{-u}{-1}.\]
The proof of the following theorem will be given in Section \ref{16proofs-sect}.

\begin{theorem}\label{HI-thm}
Each of $\Gamma$, $\Gamma'$, $\Gamma''$, and $\Gamma'''$ is an $(\H,\I)$-structure and every nontrivial $(\H,\I)$-structure is diagonally equivalent to one of these.
\end{theorem}

An $(\H,\I)$-structure  $\gamma$
defines 
an $\H$-module structure on $\cA \I$ for every Coxeter system $(W,S)$. 
Analogous to our definition for $(\H,W)$-structures, we say that $\gamma$ is \emph{pre-canonical} if each of these $\H$-modules has a pre-canonical $\H$-module structure  in which $\I$ partially ordered by the Bruhat order is the ``standard basis.'' 
We have the same remark as concerned pre-canonical $(\H,W)$-structures:
by the preceding theorem and Lemma \ref{uniquepre-lem}, if $\gamma$ is a nontrivial, pre-canonical $(\H,\I)$-structure, then for each choice of Coxeter system $(W,S)$ there is a unique bar involution $\psi : \cA \I \to \cA \I$ such that $(\psi,\I)$ is a pre-canonical $\H$-module structure, and this structure always admits a canonical basis.
 We  have this analogue of Theorem \ref{4-thm}, whose proof will be given in Section \ref{16proofs-sect}.

\begin{theorem}\label{16-thm}
Exactly 16 nontrivial $(\H,\I)$-structures are pre-canonical; in particular, each of $\Gamma$, $\Gamma'$, $\Gamma''$, and $\Gamma'''$ is pre-canonical. However,  the 16 associated pre-canonical structures on $\cA \I$ are all isomorphic (in the sense of Definition \ref{morphism-def}). 
\end{theorem}

\subsection{Generic structures for  a modified Iwahori-Hecke algebra}
\label{32-sect}

Let $\H_{2}$ be the free $\cA$-module with a basis given the symbols $ K_w$ for $w \in W$, with the unique $\cA$-algebra structure  such that 
\[  K_s  K_w = \begin{cases}  K_{sw}&\text{if }sw>w \\  K_{sw} + (v^2-v^{-2}) K_w &\text{if }sw<w\end{cases}
\qquad\text{for }s \in S\text{ and }w \in W.\]
We call this the \emph{Iwahori-Hecke algebra of $(W,S)$ with parameter $v^2$}.
We again denote by $K \mapsto \overline K$ the $\cA$-antilinear map $\H_2 \to \H_2$ with $\overline { K_w} =  K_{w^{-1}}^{-1}$ for $w \in W$.
This  ``bar involution'' together with the ``standard basis''  $\{K_w\}$ indexed by $(W,\leq)$ forms a pre-canonical structure on $\H_2$, which admits a canonical basis $\{ \underline K_w\}$.
The $\ZZ$-linear map \be\label{hh2}\Phi : \H \to \H_2\ee with 
$\Phi(v^n H_w) = v^{2n} K_w$   is an injective ring homomorphism and $\underline K_w = \Phi(\underline H_w)$ for all $w \in W$.

We adapt the definition of an $(\H,\I)$-structure to the modified Iwahori-Hecke algebra $\H_2$ in the following natural way.
Consider  a $4\times 2$ matrix $\gamma = (\gamma_{ij})$ with entries in $\cA$.
Define $\rho_{\gamma,2} : \{ K_s : s \in S \} \to \End(\cA \I)$ again by the formula  \eqref{form-eq1} except with $H_s$ replaced by $K_s$; that is, let $\rho_{\gamma,2}$ be the composition of $\rho_\gamma$ with the obvious bijection $\{ K_s : s \in S\} \to \{ H_s : s \in S\}$.

\begin{definition} 
The matrix $\gamma$ is an \emph{$(\H_2,\I)$-structure}
if for every Coxeter system $(W,S)$, the map
$\rho_{\gamma,2}$ extends to a representation of $\H_2=\H_2(W,S)$ in $\cA\I = \cA \I(W,S)$.
\end{definition}

Despite the formal similarity of this definition to Definition \ref{HI-def}, there are at least two good reasons to consider $(\H_2,\I)$-structures in addition to $(\H,\I)$-structures. First, the generic structures which have so far been uncovered ``in nature,'' through the work of Lusztig and Vogan \cite{LV2,LV1,LV6}, are in fact $(\H_2,\I)$-structures, and we will actually deduce the existence of the $(\H,\I)$-structures in the previous section from the existence of such $(\H_2,\I)$-structures. Second, we will find that a more complicated and interesting classification applies to ``pre-canonical'' $(\H_2,\I)$-structures, which does not follow directly from the theorems in Section \ref{16-sect}.

Given a  matrix $\gamma$ over $\cA$, define $[\gamma]_2$   by applying the ring endomorphism of $\cA$ with $v\mapsto v^2$ to  the entries of $\gamma$.
The following motivates this notation.

\begin{observation}\label{H2I-obs}
If $\gamma$ is an $(\H,\I)$-structure then $[\gamma]_2$ is an $(\H_2,\I)$-structure.
\end{observation}

As before, we say that an $(\H_2,\I)$ structure $\gamma$ is \emph{trivial} if $\gamma_{11} = \gamma_{21} = \gamma_{31} = \gamma_{41} =0$ and  $\gamma_{12} = \gamma_{22}=\gamma_{32} = \gamma_{42} \in \{v^2,-v^{-2}\}$. 
Lemma \ref{diag-lem} holds \emph{mutatis mutandis} with ``$(\H,\I)$-structure'' replaced by ``$(\H_2,\I)$-structure'' and ``$\H$'' replaced by ``$\H_2$.''  Define two $(\H_2,\I)$-structures to be \emph{diagonally equivalent} as in that result.
The classification of $(\H_2,\I)$-structures
up to diagonal equivalence is  no different than for $(\H,\I)$-structures:

\begin{theorem}\label{H2I-thm} Let $\Gamma$, $\Gamma'$, $\Gamma''$, and $\Gamma'''$ be the $(\H,\I)$-structures defined before Theorem \ref{HI-thm}. Then
every nontrivial $(\H_2,\I)$-structure is diagonally equivalent to $[\Gamma]_2,$ $[\Gamma']_2$, $[\Gamma'']_2$, or $[\Gamma''']_2$.
\end{theorem}

The proof of this result will be sketched in Section \ref{16proofs-sect}.
Define an $(\H_2,\I)$-structure $\gamma$ to be \emph{pre-canonical} exactly as for $(\H,\I)$-structures: namely, say that $\gamma$ is pre-canonical if, for every Coxeter system $(W,S)$, 
there exists a pre-canonical $\H_2$-module structure on $\cA \I$ (relative to the $\H_2$-module structure defined by $\gamma$)  in which $\I$ partially ordered by the Bruhat order is the ``standard basis.''
Just like for $(\H,W)$-structures and $(\H,\I)$-structures, if an $(\H_2,\I)$-structure is nontrivial and pre-canonical, then by Lemma \ref{uniquepre-lem} it associates a unique pre-canonical $\H_2$-structure to $\cA \I$ for each Coxeter system $(W,S)$.

To classify the pre-canonical $(\H_2,\I)$-structures, we define $\Delta$ and $\Delta'$ as the    matrices
\[
\Delta = \menrep{1}{0}{1}{v^2-v^{-2}}{v+v^{-1}}{1}{v-v^{-1}}{v^2-1-v^{-2}}
\qquand
\Delta' = \menrep{1}{0}{1}{v^2-v^{-2}}{v^{-1}+v}{-1}{v^{-1}-v}{v^2+1-v^{-2}}.
\]
In addition, let 
\[
\Delta'' = [\Gamma]_2
\qquand
\Delta''' = [\Gamma'']_2.
\] 
%
%
%
Observe that $\Delta$ and $\Delta''$ (respectively, $\Delta'$ and $\Delta'''$) are diagonally equivalent, which is how we deduce that $\Delta$ and $\Delta'$ are $(\H_2,\I)$-structures. Note, however, the $\H_2$-module structures defined by the pairs $\Delta$ and $\Delta''$ (respectively, $\Delta'$ and $\Delta'''$)  are technically not isomorphic, although they  would be if all of our algebras and modules were defined over the field $\QQ(v)$ instead of the ring $\cA$.
We  have this analogue of Theorems \ref{4-thm} and \ref{16-thm}. 

\begin{theorem}\label{32-thm}
Exactly 32 nontrivial $(\H_2,\I)$-structures  are pre-canonical; in particular, each of $\Delta$, $\Delta'$, $\Delta''$,  and $\Delta'''$ is pre-canonical. The 32 associated pre-canonical structures on $\cA \I$ are each isomorphic (in the sense of Definition \ref{morphism-def}) to one of the  structures arising from $\Delta$, $\Delta'$, $\Delta''$, or $\Delta'''$. 
\end{theorem}

The proof of this theorem appears at the end of Section \ref{16proofs-sect}.

\begin{remark}
 The pre-canonical structures on $\cA \I$ defined by the $(\H,\I)$-structures $\Gamma$ and $\Gamma''$ are isomorphic by Theorem \ref{16-thm}, and one might expect this to imply that  the pre-canonical structures defined by $\Delta'' = [\Gamma]_2$ and $\Delta''' = [\Gamma'']_2$ are likewise isomorphic.
The reason this does not follow is that the latter structures admit canonical bases $\{b_w\}$ and $\{b'_w\}$ of the form
$b_w = \sum_{y \leq w} f_{y,w}(v^{-2})  y $ and $ b'_w = \sum_{y\leq w} f_{y,w}(-v^{-2})  y$
for some polynomials $f_{y,w}(t) \in \ZZ[t]$.
 Corollary \ref{canon-cor} shows that   such canonical bases cannot arise from isomorphic pre-canonical structures, provided $f_{y,w}(t)$ are sufficiently complicated polynomials.
 \end{remark}

It will follow from  the discussion in Sections \ref{precanon2-sect} and \ref{precanon3-sect} (or more concretely, from small computations) that   the pre-canonical structures which $\Delta$, $\Delta'$, $\Delta''$, and $\Delta'''$ associate to  $\cA \I$ are generally not isomorphic.
Thus, we are left with the interesting question of explaining where these four structures come from. The structure $\Delta$ is what has appeared naturally from geometric considerations in the work of Lusztig and Vogan \cite{LV2,LV1,LV6}, and one can account for $\Delta''$ and $\Delta'''$ 
as the two distinct ``extensions'' of the unique isomorphism class of pre-canonical $(\H,\I)$-structures.
The origins of the remaining pre-canonical $(\H_2,\I)$-structure $\Delta'$ remains more mysterious.

\section{Existence proofs}
\label{exist-sect}

The results in Sections \ref{16-sect} and \ref{32-sect} assert, in one direction, that certain generic structures exist, or equivalently, that certain formulas define $\H$- or $\H_2$-modules structures on $\cA \I$ (and sometimes also admit compatible pre-canonical structures) for all Coxeter systems $(W,S)$.  In this section we prove some existence statements of this type, which we require for the proofs of Theorems \ref{HI-thm}, \ref{16-thm}, \ref{H2I-thm}, and \ref{32-thm} given in Section \ref{16proofs-sect}.

\subsection{A canonical basis for twisted involutions}
\label{lv-sect}


Our starting point is the following result of 
 Lusztig and Vogan, first proved in \cite{LV1} in the  case that $W$ is a Weyl group or affine Weyl group, then extended in \cite{LV2} to arbitrary Coxeter systems by elementary methods. Lusztig and Vogan's preprint \cite{LV6} provides another proof of this result, using the machinery of Soergel bimodules developed by Elias and Williamson in \cite{EW}.

\begin{thmdef}[Lusztig and Vogan \cite{LV1,LV6}; Lusztig \cite{LV2}]
\label{m1-thm}
There is a unique $\H$-module  \[\cL=\cL(W,S)\] which, as an $\cA$-module, is free with a basis given by the symbols $L_{w}$ for $w \in \I$, and which satisfies
\[ K_s L_{w} 
= 
\left\{
\ba 
L_{sws} \ {\color{white}+}\ &  				&&\ \text{if $sw \neq ws > w$} \\ 
L_{sws} \ +\  & (v^2-v^{-2})L_{w} 			&&\ \text{if $sw \neq ws < w$} \\
(v+v^{-1})L_{s w} \ +\ & L_{w}				&&\ \text{if $sw =ws >w$} \\
(v-v^{-1})L_{sw} \ +\ &(v^2-1-v^{-2}) L_{w} 	&&\ \text{if $s w =ws <w$} 
\ea
\right.
\]
for $s \in S$ and $w \in \I$.
\end{thmdef}

\begin{proof}
This is  \cite[Theorem 0.1]{LV2}, where $v^2=u$ and $K_s = u^{-1} T_s$ and $L_w = a'_w=v^{-\ell(w)} a_w$.
\end{proof}

The preceding theorem shows that the matrix $\Delta$ in Section \ref{32-sect} is an $(\H_2,\I)$-structure. The following result, which combines \cite[Theorem 0.2, Theorem 0.4, and Proposition 4.4]{LV2}, shows that this $(\H_2,\I)$-structure is pre-canonical.
Here, for $x \in W$ we write   $\sgn(x) = (-1)^{\ell(x)}$.

\begin{thmdef}[Lusztig and Vogan \cite{LV1}; Lusztig \cite{LV2}]\label{precanonstruct1-thm}
Define 
\begin{itemize}
\item the ``bar involution'' of $\cL$ to be the $\cA$-antilinear map $\cL \to \cL$,  denoted $L \mapsto \overline{L}$, with 
\[ \overline{L_{(x,\theta)} } =   \sgn(x) \cdot \overline {K_x} \cdot L_{(x^{-1},\theta)}\qquad\text{for }(x,\theta) \in \I.\]
\item the ``standard basis'' of $\cL$ to be $\{ L_w\}$ with the partially ordered index set $(\I,\leq)$.
\end{itemize}
This is a pre-canonical
 $\H_2$-module structure on $\cL$, and 
it admits a canonical basis $\{ \underline L_w\}$.
\end{thmdef}

Observe, by Lemma \ref{uniquepre-lem}, that the pre-canonical $\H_2$-module structure thus defined on $\cL$  is the unique one in which $\{ L_w\}$ serves as the ``standard basis.''
Following the convention in \cite{LV2}, we define $\pi_{y,w} \in \ZZ[v^{-1}]$ for  $y,w \in \I$ such that 
  $\underline L_w  =  \sum_{y \in \I} \pi_{y,w} L_y$. Note that $\pi_{y,w} = \delta_{y,w}$ if $y \not < w$.
 We note the following degree bound from  \cite[Section 4.9(c)]{LV2}.
 
\begin{proposition}[Lusztig \cite{LV2}]  \label{degbound2-prop} If $y,w \in \I$ such that $y \leq w$ then $v^{\ell(w)-\ell(y)}  \pi_{y,w} \in 1+ v^2\ZZ[v^2]$.
\end{proposition}
%

\begin{remark}
The  polynomials $v^{\ell(w)-\ell(y)}  \pi_{y,w}$ are denoted  $P^\sigma_{y,w}$  in \cite{LV2,LV1,EM1,EM2}.
 Lusztig proves an inductive formula \cite[Theorem 6.3]{LV2} for the action of $\underline K_s = K_s + v^{-2}\in \H_2$ on $\underline L_w$ which can be used to compute these polynomials; see \cite[Section 2.1]{EM2}.
\end{remark}

%
%
%
%

While the polynomials $\pi_{y,w}$ may have negative coefficients, they still   possess  certain positivity properties.
Recall that $h_{y,w} \in \NN[v^{-1}]$ are the polynomials such that $\underline H_w = \sum_{y \in W } h_{y,w} H_y$.
Given $y,w \in W$ and $\theta,\theta' \in \Aut(W,S)$, define $h_{(y,\theta), (w,\theta')}$ to be $h_{y,w}$ if $ \theta = \theta'$ and zero otherwise.
Lusztig \cite[Theorem 9.10]{LV2} has shown  that \[ \tfrac{1}{2}\(h_{y,w} \pm \pi_{y,w}\) \in \ZZ[v^{-1}]\qquad\text{for all $y,w \in \I$}
\]  and has  conjectured that these polynomials actually belong to $\NN[v^{-1}]$.
Lusztig and Vogan provide a geometric proof of this conjecture when $W$ is a Weyl group (see \cite[Section 3.2]{LV1}) and outline a proof for arbitrary Coxeter systems in \cite{LV6}.
The canonical basis $\{ \underline L_w\}$ conjecturally displays  
some other  positivity properties, which are considered in detail in \cite{EM1,EM2}.


%

\subsection{Another pre-canonical $\H_2$-module structure} \label{precanon3-sect}

Here we deduce from the results in the previous section that the matrix $\Delta'$ in Section \ref{32-thm} is a pre-canonical $(\H_2,\I)$-structure.
The pre-canonical $\H_2$-module structure on $\cA \I$ associated to this generic structure admits a canonical basis which  is not related in any obvious way to the basis $\{ \underline L_w\}$ in the previous section, although it has similar properties.
 It is an open problem to find an interpretation of this new canonical basis along the lines of \cite{LV1,LV6}.


First we have this analogue of Theorem-Definition \ref{m1-thm}, showing that $\Delta'$ is in fact an $(\H_2,\I)$-structure.
\begin{thmdef}\label{m3-thm}
There is a unique $\H_2$-module  
\[ \cL' =\cL'(W,S)\]
which, as an $\cA$-module, is free with a basis given by the symbols $L'_{w}$ for $w \in \I$, and which satisfies
\[ K_s L'_{w} 
= 
\left\{
\ba 
L'_{sws} \ {\color{white}+}\ &  				&&\ \text{if $sw \neq ws > w$} \\ 
L'_{sws} \ +\  & (v^2-v^{-2})L'_{w} 			&&\ \text{if $s w \neq ws < w$} \\
(v^{-1}+v)L'_{s w} \ -\ & L'_{w}				&&\ \text{if $s w =ws >w$} \\
(v^{-1}-v)L'_{sw} \ +\ &(v^2+1-v^{-2}) L'_{w} 	&&\ \text{if $sw =ws <w$} 
\ea
\right.
\]
for $s \in S$ and $w \in \I$.
\end{thmdef}

\begin{proof}
Define $f_{x,y}^z \in \cA$ for $x \in W$ and $y,z \in W$ such that $(-1)^{\rho(y)} K_x  L_y = \sum_{z \in \I} (-1)^{\rho(z)} f^z_{x,y}L_z$.
It is a straightforward exercise to check, using 
the well-known relations defining $\H_2$ (see, e.g., \cite[Section 2.1]{LV2}),  that there is  a unique $\H_2$-module structure on $\cL'$ in which $K_x L'_y = \sgn(x) \sum_{z \in \I} \overline{f_{x,y}^z} L'_z$ for $x \in W$ and $y \in \I$. In this $\H_2$-module structure, the generators $K_s$ for $s \in S$ act on the basis elements $L'_w$ according to the given formula.
\end{proof}

We have this analogue of Theorem-Definition \ref{precanonstruct1-thm}, which shows that $\Delta'$ is pre-canonical.

\begin{thmdef}\label{preL'-thm}
Define 
\begin{itemize}
\item the ``bar involution'' of $\cL'$ to be the $\cA$-antilinear map $\cL'\to\cL'$, denoted $L' \mapsto \overline{L'}$, with 
\[ \overline{L'_{(x,\theta)} } =    \overline {K_x} \cdot L'_{(x^{-1},\theta)}\qquad\text{for }(x,\theta) \in \I.\]
\item the ``standard basis'' of $\cL'$ to be $\{ L'_w\}$ with the partially ordered index set $(\I,\leq)$.
\end{itemize}
This is a pre-canonical
 $\H_2$-module structure on $\cL'$, and 
it admits a canonical basis $\{ \underline L'_w\}$. 
\end{thmdef}

By Lemma \ref{uniquepre-lem}, this is the unique pre-canonical $\H_2$-module structure on $\cL'$ in which $\{ L'_w\}$ is the ``standard basis.''

\begin{proof}
Define $r_{y,w} \in \cA$ for $y,w \in I$ such that $\overline{L_w} = \sum_{y\in \I} r_{y,w} L_y$ and let $f_{x,y}^z$ 
be as in the proof of Theorem-Definition \ref{m3-thm}.
Let $L \mapsto \wt L$ be the $\cA$-antilinear map with $\wt {L'_w} = \sum_{y \in \I} (-1)^{\rho(w)-\rho(y)}\cdot \overline{r_{y,w}}\cdot L'_y$ for $w \in \I$. 
We claim that $\wt L = \overline L$ for all $L \in \cL'$.
To prove this, we note that if $w = (x,\theta) \in \I$ then
\[ K_{x^{-1}} \wt {L'_w}= \sgn(x) \sum_{y\in \I} \sum_{z \in \I} (-1)^{\rho(z)-\rho(y)} \cdot \overline{r_{y,w}\cdot  f_{x^{-1},y}^z} \cdot L'_z
\]
while  
\[ L_w  = \sgn(x) K_{x^{-1}} \overline{L_{w}} = \sgn(x) \sum_{y\in \I} \sum_{z \in \I} (-1)^{\rho(z)-\rho(y)} \cdot r_{y,w} \cdot f_{x^{-1},y}^z\cdot  L_z.
\]
We deduce that $K_{x^{-1}} \wt {L'_w} = L'_w = K_{x^{-1}} \overline{L'_w}$ since the right side of the first equation is the image of the right side of the second under the $\cA$-antilinear map $\cL \to \cL'$ with $L_z \mapsto  L'_z$ for $z \in \I$.
Since $K_{x^{-1}}$ is invertible
this shows that $\wt {L'_w} = \overline{ L'_w} $ 
for $w \in \I$ which suffices to prove our claim.

Given the claim, it follows from Theorem-Definition \ref{precanonstruct1-thm} that the bar involution and standard basis of $\cL'$ form a pre-canonical structure, and it is easy to show that the identity $\overline{K_s L_w} = \overline{K_s} \cdot \overline{L_w}$ implies $\overline{K_s L'_w} = \overline{K_s} \cdot \overline{L'_w}$ for $s \in S$ and $w \in \I$.
Hence the bar involution and standard basis of $\cL'$ form a pre-canonical $\H_2$-module structure, which admits a canonical basis $\{ \underline L'_w\}$ by Theorem \ref{canon-thm}.
\end{proof}

We spend the rest of this section establishing a few  properties of the canonical basis $\{ \underline L'_w\}$. 
Define $\pi'_{y,w} \in \ZZ[v^{-1}]$ for  $y,w \in \I$ as the polynomials such that 
  $\underline L'_w  =  \sum_{y \in \I} \pi'_{y,w} L'_y$. 
  We introduce some notation to state a recurrence for computing these polynomials.
First, for $y,w \in \I$ let
 \ben
 \item[] $  \mu'(y,w)
  = (\text{the coefficient of $v^{-1}$ in $\pi'_{y,w}$}),
$
\item[]
$
  \mu''(y,w)
  =   (\text{the coefficient of $v^{-2}$ in $\pi'_{y,w}$})+(v+v^{-1}) \mu'(y,w) .
  $
  \een
Next, for $s \in S$ and $y,w \in \I$ define
\ben
\item[] $\mu'(s,y,w) =  \delta_{sy<y}\cdot \mu''(y,w) + \delta_{sy,ys} \cdot (\ell(y)-\ell(sy)) \cdot \mu'(sy,w)
  - \ds\sum_{\substack{y< z < w \\ sz<z}} \mu'(y,z)\mu'(z,w).$
  \een
  Here $\delta_{sy<y}$ is 1 if $sy<y$ and 0 otherwise. 
 In what follows,  recall that $\underline K_s = K_s + v^{-2}$ for $s \in S$.

  \begin{proposition}\label{L'rec-prop} Let $w \in \I$ and $s \in S$ be such that $w<sw$.
  \ben
  \item[(a)] If $sw\neq ws$ then $\underline K_s \underline L'_w =  \underline L'_{s ws} + \sum_{y<sws} \mu'(s,y,w) \underline L'_y$.
  
    \item[(b)] If $sw= ws$ then $\underline K_s \underline L'_w
    =
    (v+v^{-1}) \underline L'_{s w} - \underline L'_w + \sum_{y<sw}  (\mu'(s,y,w)-\mu'(y,sw)) \underline L'_y
$.
    
  \een
  \end{proposition}
  
  \begin{remark}
 Lusztig  \cite[Theorem 6.3(c)]{LV2}
  shows that the canonical basis $\{\underline L_w\} \subset \cL$  in the previous section is such that $\underline K_s \underline L_w = (v^2+v^{-2}) \underline L_{ w}$ if $s \in S$ and $w \in \I$ and $sw<w$. This property has no simple analogue for the canonical basis $\{ \underline L'_w\} \subset \cL'$.
  \end{remark}
  
  \begin{proof}
Each part of the proposition follows by showing that  the difference between the two sides of the desired equality both (i) is an element of the set $\sum_{y < s\act w} v^{-1} \ZZ[v^{-1}] \cdot L'_y$
and (ii) is invariant under the bar operator of $\cL'$. Since the only such element with these two properties is 0, the given identities must hold. The observation (ii) follows in either case from Theorem-Definition \ref{preL'-thm}, while showing that property (i) holds is a straightforward exercise from Theorem-Definition \ref{m3-thm}.
  \end{proof}

Write $f \equiv g \modu 2)$ if $f,g \in \cA$ are such that $f-g \in 2 \cA$,
and define $\pi_{y,w}$ and $h_{y,w}$ for $y,w \in \I$ as in the previous section.
 We note the following relationship between  $\pi'_{y,w}$, $\pi_{y,w}$, and $h_{y,w}$.
 
  \begin{proposition}\label{congr-prop} For all $y,w \in \I$ it holds that $\pi'_{y,w} \equiv \pi_{y,w} \equiv h_{y,w} \modu 2)$.
  \end{proposition}
  
  \begin{proof}
The second congruence is \cite[Theorem 9.10]{LV2}. For $F \in \cL$ and $G \in \cL'$, we write $F \equiv G\modu 2)$ if $F = \sum_{y \in \I} f_y L_y$ and $G = \sum_{y \in \I} g_y L'_y$ for some polynomials $f_y,g_y \in \cA$ with $ f_y \equiv g_y \modu 2)$ for all $y \in \I$. To prove the first congruence we must show that $\underline L_w \equiv \underline L'_w \modu 2)$ for all $w \in \I$. This automatically holds if $\rho(w) = 0$. Let  $w \in \I$ and $s \in S$ be such that $w<sw$ and assume $\underline L_y \equiv \underline L'_y \modu 2)$ if $y<s\act w$.
It suffices to show under this hypothesis that
\be\label{congr-eq}\underline L_{s\act w} \equiv \underline L'_{s\act w} \modu 2).\ee
Towards this end, define $\mu(y,w) \in \ZZ$ for $y,w \in \I$ as the coefficient of $v^{-1}$ in $\pi_{y,w}$, and let
  \[ X_{s, w} = \begin{cases} \underline L_{sws}&\text{if }sw\neq ws \\
  (v+v^{-1}) \underline L_{sw} - \sum_{y<sw} \mu(y,sw) \underline L_y&\text{if }sw=ws
  \end{cases}
\]
and
\[
  X'_{s, w} = \begin{cases} \underline L'_{sws}&\text{if }sw\neq ws \\
  (v+v^{-1}) \underline L'_{sw} - \sum_{y<sw} \mu'(y,sw) \underline L'_y&\text{if }sw=ws.
  \end{cases}
  \]
 We claim that to prove the congruence \eqref{congr-eq} it is enough show that $X_{s, w} \equiv X'_{s, w} \modu 2)$.
 This is obvious if $sw\neq ws$ so assume $sw= ws$ and $X_{s, w} \equiv X'_{s, w} \modu 2)$. We must check that 
 $\pi'_{y,s w} \equiv \pi_{y,s w} \modu 2)$ for all $y \leq sw$; 
to this end we argue by induction on $\rho(sw)-\rho(y)$. By definition $\pi'_{s w,s w} = \pi_{s w,s w} = 1$. Fix $y < sw$ and suppose $\pi'_{z,s w} \equiv \pi_{z,s w} \modu 2)$ for $y <z\leq sw$.
 The congruence $X_{s, w} \equiv X'_{s, w} \modu 2)$ implies
  \[ (v+v^{-1}) \pi_{y,sw} - \sum_{y\leq z<sw} \mu(z,sw)\pi_{y,z} \equiv 
 (v+v^{-1}) \pi'_{y,sw} -  \sum_{y\leq z<sw} \mu'(z,sw)\pi'_{y,z} \modu 2).\]
By hypothesis, 
 the terms indexed by $z>y$ in the sums on either side of this congruence cancel, and we obtain
  \[ (v+v^{-1}) \pi_{y,sw} -  \mu(y,sw)\equiv 
 (v+v^{-1}) \pi'_{y,sw} -  \mu'(y,sw) \modu 2).\] 
 It is an elementary exercise, noting that $\pi_{y,sw}$ and $ \pi'_{y,sw} $ both belong to $ v^{-1}\ZZ[v^{-1}]$, to show that this congruence implies $\pi_{y,sw} \equiv \pi'_{y,sw} \modu 2)$, and so  we conclude by induction that \eqref{congr-eq} holds.
  This proves our claim.
  
  We now argue that $X_{s, w} \equiv X'_{s, w}\modu 2)$.
For this we observe that there are unique polynomials $a_{s,y,w},a'_{s,y,w} \in \cA$  such that
\[
  X_{s, w}= \underline K_s \underline L_w- \sum_{y<s\act w} a_{s,y,w} \underline L_y
\qquand
  X'_{s, w} = \underline K_s \underline L'_w- \sum_{y<s\act w} a'_{s,y,w} \underline L'_y.
\]
Indeed, the polynomials $a'_{s,y,w}$ are given by Proposition \ref{L'rec-prop},
and an entirely analogous statement decomposing the product $\underline K_s \underline L_w$ gives the polynomials $a_{s,y,w}$. It is not difficult to show, by deriving a formula for $a_{s,y,w}$ similar to the one for $\mu'(s,y,w)$, that the  hypothesis $\underline L_y \equiv \underline L'_y \modu 2)$ for $y<s\act w$ implies $a_{s,y,w} \equiv a'_{s,y,w} \modu 2)$. Hence to prove $X_{s, w} \equiv X'_{s, w}\modu 2)$
we need only check that $\underline K_s \underline L_w \equiv \underline K_s \underline L'_w \modu 2)$. As we assume $\underline L_w \equiv \underline L'_w \modu 2)$, this follows by comparing Theorem-Definitions \ref{m1-thm} and \ref{m3-thm}, which 
shows more generally that $\underline K_s F \equiv \underline K_s G \modu 2)$ whenever $F \in \cL$ and $G \in \cL'$ such that $F \equiv G \modu 2)$.
  \end{proof}

   The polynomials $\pi'_{y,w}$ also satisfy the same degree bound as $\pi_{y,w}$ and $h_{y,w}$.
   
  \begin{proposition}\label{deg3-prop} If $y,w \in \I$ such that $y \leq w$ then $v^{\ell(w)-\ell(y)}  \pi'_{y,w} \in 1+ v^2\ZZ[v^2]$.
\end{proposition}

\begin{proof}
The proposition holds if $\rho(w) = 0$ since then $\pi'_{y,w} = \delta_{y,w}$. 
Let $w \in \I$ and $s \in S$ be such that $w<sw$ and assume $v^{\ell(z)-\ell(y)}  \pi'_{y,z} \in 1+ v^2\ZZ[v^2]$ for all $y \leq z < s\act w$.
It suffices   to show   under this hypothesis that
\be\label{v^eq} v^{\ell(s\act w) - \ell(y)} \pi '_{y,s\act w} \in 1 + v^2 \ZZ[v^2]\qquad\text{for all $y \in \I$ with $y\leq s\act w$.}\ee
To this end, define $X'_{s, w}$ as in the proof of Proposition \ref{congr-prop} and let 
$p_y \in \cA$ for $y \in \I$ be such that 
$ X'_{s, w} = \sum_{y \leq s\act w} p_y L'_y.$
We claim that to prove \eqref{v^eq} it is enough to show that 
\be\label{u^eq} v^{\ell(w) - \ell(y)  + 2} p_y \in 1 + v^2 \ZZ[v^2]\qquad\text{for all $y \in \I$ with $y \leq s\act w$.}\ee
 This follows when $sw\neq ws$ as then $\ell(s\act w) = \ell(w) +2$ and $p_y = \pi'_{y,s\act w}$. Alternatively, suppose that  $sw=ws$ and 
 \eqref{u^eq} holds. We then have
\be\label{py-eq}  p_y =(v+v^{-1}) \pi'_{y,sw} -\mu'(y,sw) - \sum_{y < z < sw} \mu'(z,sw) \pi'_{y,z}.\ee
To deduce  \eqref{v^eq}, we argue by induction on $\ell(s w) -\ell(y)$. If $y =s w$ then the desired containment holds automatically. Let $y < s w$ and suppose $v^{\ell(s w) - \ell(z)} \pi'_{z,s w} \in 1 + v^2 \ZZ[v^2]$ for $y <z \leq s w$.
Then  $\mu'(z,sw)$   is nonzero for $z>y$ only if $\ell(w)-\ell(z)$ is even, 
so if we multiply both sides of \eqref{py-eq} by $v^{\ell(w)-\ell(y)+2}$, then it follows from
 \eqref{u^eq} via our inductive hypothesis that 
 \[      (v^2+1) v^{\ell( s  w) - \ell(y)}  \pi'_{y,s  w} - v^{\ell( s  w) - \ell(y)+1}   \mu'(y,sw) \in 1+ v^2 \ZZ[v^2].\]
 Since we always have $\pi'_{y,sw} \in v^{-1}\ZZ[v^{-1}]$ and $\mu'(y,sw) \in \ZZ$, this containment can only hold if $\mu'(y,sw) = 0$ whenever $\ell(sw)-\ell(y)$ is even. We deduce from this that in fact
 \[   (v^2+1) v^{\ell( s  w) - \ell(y)}  \pi'_{y,s  w} \in 1 + v^2 \ZZ[v^2]\]
 and it is easy to see that this implies $v^{\ell( s  w) - \ell(y)}  \pi'_{y,s  w} \in 1 + v^2 \ZZ[v^2]$, which is what we needed to show. We conclude by induction that \eqref{u^eq} implies \eqref{v^eq}.

We now argue that \eqref{u^eq} holds.  
Fix $y \leq s\act w$.
Proposition \ref{L'rec-prop}
then implies
\[ 
p_y = (a+\delta_{sw,ws}) \cdot \pi_{y,w}' + b\cdot \pi'_{s\act y,w} - \Sigma
\]
where 
\[
(a,b)
=
\begin{cases}
 (v^{-2},\ 1) &\text{if  $sy\neq ys>y$}
 \\
 (v^{2},\ 1) &\text{if  $sy\neq ys<y$}
 \\
 (v^{-2}-1,\  v^{-1}-v) &\text{if  $sy= ys>y$}
 \\
 (v^{2}+1,\ v^{-1}+v) &\text{if  $sy= ys<y$}
 \end{cases}
 \qquand
 \Sigma = \sum_{z<s\act w} \mu'(s,z,w)\pi'_{y,z}.
 \]
Since we  assume that 
  $v^{\ell(z')-\ell(z)} \pi'_{z,z'}  \in 1+ v^2\ZZ[v^2]$ for  $z\leq z'  \leq w$, inspecting our definition shows that $\mu'(s,z,w) $ 
is an integer when $\ell(w)-\ell(z)$ is even and an integer multiple of $v+v^{-1}$ when $\ell(w)-\ell(z)$ is odd. 
Consequently, it follows that 
\[v^{\ell(w)-\ell(y)+2} \Sigma \in v^2 \ZZ[v^2].\]
In turn, since $y \leq s\act w$,  \cite[Lemma 2.7]{H3} implies that $s\act y \leq w$ if $sy<y$ and that
$y \leq w$ if $sy>y$. Using this fact and the hypothesis stated in the second sentence of this proof, one checks that 
\[ v^{\ell(w)-\ell(y)+2} \( (a+\delta_{sw,ws})\cdot  \pi'_{y,w} + b\cdot \pi'_{s\act y,w} \)\in 1 + v^2 \ZZ[v^2].\]
Combining these observations, we conclude  that \eqref{u^eq} holds.
\end{proof}

Despite these results,
there does not appear to be any simple relationship between the polynomials $\pi_{y,w}$ and $\pi'_{y,w}$, and it is unclear what positivity properties the latter polynomials possess, if any. In general, $\pi'_{y,w}$ may have both positive and negative coefficients. The combination of Propositions \ref{degbound1-prop}, \ref{degbound2-prop}, \ref{congr-prop}, and \ref{deg3-prop} shows that 
\be\label{hpi'-eq} \tfrac{1}{2}\( h_{y,w} \pm \pi'_{y,w}\) \qquand \tfrac{1}{2}\(\pi_{y,w} \pm \pi'_{y,w}\)\ee
are polynomials in $v^{-1}$ with integer coefficients, which become polynomials in $v^2$ when multiplied by $v^{\ell(w)-\ell(y)}$. Unlike the analogous polynomials 
$\tfrac{1}{2}\( h_{y,w} \pm \pi_{y,w}\)$ discussed at the end of the previous section (which conjecturally belong to $\NN[v^{-1}]$),  the four polynomials in \eqref{hpi'-eq} can each have both positive and negative coefficients.

\subsection{A third canonical basis for twisted involutions} \label{precanon2-sect}


We finally prove here that the matrix $\Gamma$ from Section \ref{16-sect} is a pre-canonical $(\H,\I)$-structure.
This provides us with another a canonical basis indexed by the twisted involutions in a Coxeter group, but not related in any transparent way to our other bases $\{ \underline L_w\}$ and $\{ \underline L'_w\}$.
It is  an open problem to find an interpretation of this third basis.

\begin{thmdef}\label{m2-thm}
There is a unique $\H$-module  
\[ \cI =\cI(W,S)\]
which, as an $\cA$-module, is free with a basis given by the symbols $I_{w}$ for $w \in \I$, and which satisfies

\[ H_s I_{w} 
= 
\left\{
\ba 
I_{sws} \ {\color{white}+}\ &  			&&\ \text{if $s\act w =sws > w$} \\ 
I_{sws} \ +\  & (v-v^{-1})I_{w} 			&&\ \text{if $s\act w =sws < w$} \\
I_{s w} \ +\ & I_{w}					&&\ \text{if $s\act w =sw >w$} \\
(v-v^{-1})I_{sw} \ +\ &(v-1-v^{-1}) I_{w} 	&&\ \text{if $s\act w =sw <w$} 
\ea
\right.
\]
for $s \in S$ and $w \in \I$.
\end{thmdef}

\begin{proof}
Define $J_w = (v+v^{-1})^{2\rho(w)-\ell(w)}  L_w \in \cL$ and let $\cJ = \ZZ[v^2,v^{-2}]\spanning\{J_w : w \in \I\}$. 
Define
 $\phi : \cI \to \cJ$ as the $\ZZ$-linear bijection with $v^{n} I_w \mapsto v^{2n} J_w$ for $w \in \I$. 
 With  $\Phi : \H \to \H_2$  the ring homomorphism \eqref{hh2}, the multiplication formula $H I =\phi^{-1}( \Phi(H) \phi(I))$ for $H \in \H$ and $I \in \cI$ makes $\cI$ into an $\H$-module, and one checks that relative to this structure the action of $H_s$ on $I_w$ is described by precisely the given formula. This $\H$-module structure is unique since the elements $H_s$ for $s \in S$ generate $\H$ as an $\cA$-algebra.
\end{proof}


Theorem is equivalent to the assertion that $\Gamma$ defined before Theorem \ref{HI-thm} is an $(\H,\I)$-structure.
In turn we have this analogue of Theorem-Definitions \ref{precanonstruct1-thm} and \ref{preL'-thm} showing that $\Gamma$ is pre-canonical.

\begin{thmdef}\label{precanonstruct2-thm}
Define 
\begin{itemize}
\item the ``bar involution'' of $\cI$ to be the $\cA$-antilinear map $\cI \to \cI$, denoted $I \mapsto \overline{I}$,
with
\[ \overline{I_{(x,\theta)} } =  \sgn(x) \cdot \overline {H_x} \cdot I_{(x^{-1},\theta)}
\qquad
\text{for $(x,\theta) \in \I$.}
\]

\item the ``standard basis'' of $\cI$ to be $\{ I_w\}$ with the partially ordered index set $(\I,\leq)$.
\end{itemize}
This is a pre-canonical $\H$-module structure on $\cI$ and it admits a canonical basis $\{ \underline I_w\}$. 
\end{thmdef}

Again by Lemma \ref{uniquepre-lem}, this is the unique pre-canonical $\H$-module structure on $\cI$ in which $\{ I_w\}$ serves as the ``standard basis.''

\begin{proof}
Define $\cJ$ and $\Phi : \H \to \H_2$ and $\phi : \cI \to \cJ$ as in the proof of Theorem-Definition \ref{m2-thm}.
The bar involution given in Theorem-Definition \ref{precanonstruct1-thm} for $\cL$ restricts to an $\cA$-antilinear map $\cJ \to \cJ$. Denote this restricted map by $ \psi'$, and write $\psi : I \mapsto \overline{I}$ for the bar involution  of $\cI$.
Since $\Phi(\overline {H_x}) = \overline{K_x}$ for all $x \in W$, it follows that $\psi = \phi^{-1}\circ \psi' \circ \phi$, and from this identity the claim that $(\psi,\{ I_w\})$ is a pre-canonical $\H$-module structure on $\cI$ follows as a consequence of Theorem-Definition \ref{precanonstruct1-thm}.
Given this, we conclude that 
a canonical basis $\{ \underline I_w\}$ exists by Theorem \ref{canon-thm}.
\end{proof}

\begin{remark}
Suppose $(W',S')$ is a Coxeter system such that $W=W'\times W'$ and $S = S' \sqcup S'$. 
Let $\theta \in \Aut(W,S)$ be the automorphism with $\theta(x,y) = (y,w)$. 
There is then an injective $\cA$-module homomorphism $\H(W',S') \to \cI(W,S)$
with 
\[ H_w \mapsto I_{((w,w^{-1}),\theta)} \qquad\text{which also maps}\qquad \underline H_w \mapsto \underline I_{((w,w^{-1}),\theta)}\qquad\text{for }w \in W'.\]
Via this map, one may view the canonical basis of $\cI$ as a generalization of the Kazhdan-Lusztig basis of $\H$. The canonical bases of $\cL$ and $\cL'$  generalize the canonical basis of $\H_2$ in an entirely analogous fashion.
\end{remark}

Define $\iota_{y,w} \in \ZZ[v^{-1}]$ for $y,w \in \I$  such that 
  $\underline I_w  =  \sum_{y \in \I} \iota_{y,w} I_y$
  and 
  let
 \[ 
\nu(s,y,w) = 
\begin{cases} 
\text{the coefficient of $v^{-1}$ in $\iota_{y,w}$} & \text{if $sy<y$}
\\
\text{the coefficient of $v^{-1}$ in $\iota_{sy,w}$} &\text{if }sy=ys>y 
\\
0&\text{otherwise}
\end{cases}
\qquad\text{for $s \in S$ and $y,w \in \I$.}
\]
Recall that $\underline H_s = H_s + v^{-1}$ for $s \in S$.
\begin{proposition}\label{rec2-prop}
If $s \in S$ and $w \in \I$ such that $w < sw$ then 
\[
 \underline H_s \underline I_w 
 =\underline I_{s\act w}
+
   \delta_{sw,ws} \underline I_w
   +  \sum_{  y< w} \nu(s,y,w) \underline I_y
.
\]
\end{proposition}

\begin{remark}
Unlike the canonical basis $\{\underline L_w\}$ (see the remark after Proposition \ref{L'rec-prop}), there is no simple formula for $\underline H_s \underline I_w$ when $s \in S$ such that $sw<w$.
\end{remark}

\begin{proof}
The difference between the two sides of the desired identity
is  invariant under the bar involution of $\cI$ and is also an element of  $\sum_{y < s\act w} v^{-1} \ZZ[v^{-1}] \cdot I_y$, as is straightforward to check from the definition of $\nu(s,y,w)$ and Theorem-Definition \ref{m2-thm}.
The only such element in $\cI$  is 0.
\end{proof}

%
%
%
%
%

We note one other proposition. 
Recall the definition of $\rho : \I \to \NN$ from Theorem \ref{rho-def}.

\begin{proposition}\label{degbound3-prop} If $y,w \in \I$ such that $y\leq w$ then $v^{\rho(w) - \rho(y)} \iota_{y,w} \in 1 + v \ZZ[v]$.
\end{proposition}
 
 \begin{proof}
 The proof is by induction on $\rho(w)$ using Proposition \ref{rec2-prop}. We omit the details, which are similar to  and somewhat simpler than those in the proof of Proposition \ref{deg3-prop}.
 \end{proof}

%
%
%
%
%
%
%

Computations indicate that there is no obvious relationship between the polynomials $\iota_{y,w}$ and the other polynomials $h_{y,w}, \pi_{y,w}, \pi'_{y,w} \in \ZZ[v^{-1}]$ we have seen so far.
For example, suppose $|S| = 2$ so that $(W,S)$ is a dihedral Coxeter system. Then  the values of 
$v^{\ell(w)-\ell(y)}h_{y,w}$ (for $y,w \in W$) and   $v^{\ell(w)-\ell(y)} \pi_{y,w}$ (for $y,w \in \I$) are all 0 or 1; see \cite[Theorem 4.3]{EM2}.
However, the polynomials $v^{\rho(w)-\rho(y)} \iota_{y,w} $ for $y,w \in \I$ can achieve any of the values $0$, $1$, $1+v$, $1-v$, or $1-v^2$. 
The polynomials $\iota_{y,w}$ may thus have negative coefficients, and do not in general satisfy any parity condition analogous to Proposition \ref{congr-prop}.

This means that the pre-canonical structure on $\cI$  
does not arise from a $P$-kernel,  since by the preceding proposition it is not in the image of the bijection in Theorem \ref{Pkernel-prop} for any choice of function $r : \I \to \ZZ$.
By contrast, it follows from \cite[\S2]{KL}, \cite[Proposition 4.4(b)]{LV2}, and the proof of Theorem-Definition \ref{preL'-thm}, respectively, that the pre-canonical structures on $\H$, $\cL$, and $\cL'$ are all in the image of this bijection relative to the function $r= \ell$ and so correspond to $P$-kernels.

 \section{Uniqueness proofs}
\label{uniqueproof-sect}

In this  section we at last give the proofs to the main results in Sections \ref{16-sect} and \ref{32-sect}.
Throughout, we recall our earlier definitions of $(\H,\I)$- and $(\H_2,\I)$-structures, and what it means for such structures to be pre-canonical.

\subsection{Proofs for results on generic structures}
\label{16proofs-sect}

We first prove Theorem \ref{HI-thm}, classifying all nontrivial $(\H,\I)$-structures, after stating two lemmas. 
Denote by $\Theta$
the
$\cA$-algebra automorphism of $\H$
with
 $\Theta(H_s) = -H_s + v-v^{-1}$ for $s \in S$. Observe that more generally $\Theta(H_w) = \sgn(w) \cdot \overline{H_w}$ for $w \in W$.

\begin{lemma}\label{theta-lem} The involution of the set of $4\times 2$ matrices with entries in $\cA$ given by the map
\[
\Theta :  \menrep{A}{B}{C}{D}{E}{F}{G}{H}
\mapsto  \menrep{-A}{v+v^{-1}-B}{-C}{v+v^{-1}-D}{-E}{v+v^{-1}-F}{-G}{v+v^{-1}-H}
\]
restricts to an involution of the set of $(\H,\I)$-structures.
\end{lemma}

\begin{proof}
 Observe that if $\gamma$ is an $(\H,\I)$-structure then
  $\rho_{\Theta(\gamma)} $ is the $\H$-representation $\rho_{\gamma}\circ \Theta$.
\end{proof}

The next lemma is more technical.
Fix a choice of parameters $A,B,C,D,E,F,G,H \in \cA$ and define $\gamma$ as in Lemma \ref{diag-lem}.

\begin{lemma}\label{eqlem1}
If $\gamma$ is an $(\H,\I)$-structure then 
the following properties hold:
\ben
\item[(a)] $(B-v)(B+v^{-1}) =  (D-v)(D+v^{-1}) = -AC$.
\item[(b)] $(F-v)(F+v^{-1}) = (H-v)(H+v^{-1}) = -EG$.
\item[(c)] If $A$ or $C$ is nonzero, then $B+D = v-v^{-1} $ and $D-H \in \{\pm 1\}$. 
\item[(d)] If $E$ or $G$ is nonzero, then $F+H = v-v^{-1}$ and $B-F \in \{\pm 1\}$.

\item[(e)] If $A,C,E,G$ are all nonzero, then $B \in \{0,v-v^{-1}\}$.
\een
\end{lemma}
\begin{proof}
In this proof we abbreviate by letting $\rho = \rho_\gamma$.
Suppose $s,t \in S$ are such that $st$ has order 3.
 Since $\rho$ defines a representation of $\H$, we  have $(\rho(H_s)-v)(\rho(H_s)+v^{-1}) w = 0$ for all $w \in \I$.
Expanding the left side of this identity for the elements $w \in \{1,s,t,sts\} \subset W \cap \I$ yields the equations in parts (a) and (b), and also the identities
\[ X(B+D+v^{-1}-v) = 0 \qquand Y (F+H+v^{-1}-v) = 0\]
for $X \in \{A,C\}$ and $Y \in \{E,G\}$. It follows that if $A$ or $C$ is nonzero then $B+D=v-v^{-1}$ and that if $E$ or $G$ is nonzero then $F+H = v-v^{-1}$.

We also  have
$\rho(H_s)\rho(H_t)\rho(H_s) w =  \rho(H_t)\rho(H_s)\rho(H_t) w$ for all $w \in \I$. Expanding both sides of this identity for $w \in \{1,s,t,sts\} \subset W \cap \I$ and then comparing coefficients yields the identities
\be\label{XY2}X(D^2+(B-D)H-EG) = 0
\qquand
Y(F^2+B(H-F)-AC) = 0
\ee
again for $X \in \{A,C\}$ and $Y \in \{E,G\}$. Assume $A$ or $C$ is nonzero, so that we can take $X$ to be nonzero. Then  $B-D = v-v^{-1}-2D$ and $-EG = (H-v)(H+v^{-1})$. Substituting these identities into the first equation in \eqref{XY2} and   dividing both sides by $X$ produces the equation
\[ D^2+(v-v^{-1}-2D)H + (H-v)(H+v^{-1})=0.\]
The left hand sides simplifies to the expression $(D-H)^2-1$, and thus $D-H \in \{ \pm 1\}$.
This establishes part (c). In a similar way one finds that if $E$ or $G$ is nonzero then $B-F \in \{\pm1\}$, which establishes part (d).

To prove part (e), suppose now that $s,t \in S$ are such that $st$ has order 4.
Then
 $(\rho(H_s)\rho(H_t))^2 w = (\rho(H_t)\rho(H_s))^2w$ for all $w \in \I$. Expanding both sides of this equation for $w=1$ and comparing the coefficients of $sts$ yields the identity $AE(DF+BH-EG)=0$.
Assume $A$, $C$, $E$,  $G$ are all nonzero. Then, after dividing both sides by $AE$ and applying the substitutions $D = v-v^{-1}-B$ and $H = v-v^{-1}-F$ and $-EG = (F-v)(F+v^{-1})$, our previous identity becomes
\[ 
(v-v^{-1}-B)B + (B-F)^2-1
= 0.\]
Since $(B-F)^2-1= 0$ by part (d), either $B =0 $ or $B=v-v^{-1}$, as claimed.
\end{proof}

\begin{proof}[Proof of Theorem \ref{HI-thm}]
We first show that $\Gamma$, $\Gamma'$, $\Gamma''$,  $\Gamma'''$ are all $(\H,\I)$-structures. The matrices $\Gamma$ and $\Gamma'''$ are  $(\H,\I)$-structures since the corresponding representations  just describe the action of $\H$ on the respective bases $\{I_w\}$ and  $\{  \overline{I_w} \}$ of $\cI$, as defined in Theorem-Definition \ref{m2-thm}.
The matrices 
$\Gamma'$ and $\Gamma''$ are  $(\H,\I)$-structures by Lemmas \ref{diag-lem} and \ref{theta-lem},
since $\Gamma' = \Theta(\Gamma)[-1,-1]$ and $\Gamma''  = \Theta(\Gamma''')[-1,-1]$.

Fix a choice of parameters $A,B,C,D,E,F,G,H \in \cA$ and define the $4\times 2$ matrix $\gamma$ as in Lemma \ref{diag-lem}. Assume $\gamma$ is an $(\H,\I)$-structure. We  show that $\gamma$ is diagonally equivalent to $\Gamma$, $\Gamma'$, $\Gamma''$, or $\Gamma'''$. 
There are four cases to consider:
\begin{itemize}

\item Suppose $AC=EG = 0$. Then $B,D,F,H \in \{ -v^{-1},v\}$ by Lemma \ref{eqlem1}, and by Lemma \ref{diag-lem} we may assume that $A,C,E,G\in \{0,1\}$. There are 144 choices of parameters satisfying these conditions. With the aid of the computer algebra system  {\sc{Magma}}, we have checked that the only matrices $\gamma$ of this form which are  $(\H,\I)$-structures
 are  the two trivial ones. (For this calculation, it suffices just to consider finite Coxeter systems of rank three.)

\item Suppose $AC \neq 0$ and $EG=0$.  By Lemma \ref{diag-lem} we may then assume that $E,G \in \{0,1\}$. By the  second and third parts of Lemma \ref{eqlem1}, it follows that $F,H \in \{-v^{-1},v\}$ and $D \in \{H\pm 1\}$ and $B = v-v^{-1}-D$.
By Lemma \ref{diag-lem} and the first part of Lemma \ref{eqlem1}, finally, we may assume that $A=1 $ and $C = -(D-v)(D+v^{-1}) \neq 0$.
This leaves  8 possible choices of parameters,
and we have  checked (again with the help of a computer)
that for each of the
 resulting matrices $\gamma$, there are finite Coxeter systems $(W,S)$ for which $\rho_\gamma$ fails to define an $\H(W,S)$-representation. Hence it cannot occur that $AC \neq 0$ and $EG =0$.

\item It follows by similar consideration 
that it cannot happen that 
$AC = 0$ and $EG \neq 0$. 

\item Finally suppose $AC \neq 0$ and $EG \neq 0$ so that $A,C,E,G$ are all nonzero. 
By  Lemma \ref{eqlem1} we then have $B \in \{0,v-v^{-1}\}$ and $D=v-v^{-1}-B$ and  $F  \in \{B\pm 1\}$ and $H=v-v^{-1}-F$ and 
$AC = 1$ and $EG \in \{ \pm (v-v^{-1}\}$; more specifically, Lemma \ref{eqlem1}  implies that $EG = v-v^{-1}$ when $B=0=F-1$ or $B=v-v^{-1} = F-1$ while in all other cases $EG = v^{-1}-v$. 
There are thus four choices for the quadruple $(B,D,F,H)$ and it is easy to see by Lemma \ref{diag-lem} that in each case $\gamma$ is diagonally equivalent to one of
 $\Gamma$, $\Gamma'$, $\Gamma''$, or $\Gamma'''$.
  \end{itemize}
  This  completes the proof of the theorem.
   \end{proof}


The property of an $(\H,\I)$-structure being pre-canonical is preserved under the operations in Lemmas \ref{diag-lem} and \ref{theta-lem}, in the following precise sense.

\begin{lemma}\label{l1-lem}
If $\gamma$ is a nontrivial, pre-canonical $(\H,\I)$-structure, then so is $\gamma[-1,-1]$, and the (unique) associated pre-canonical structures on $\cA \I$ are isomorphic via the identity map,
which has as a scaling factor
 the $\cA$-linear map $\cA\I \to \cA \I$ with $w \mapsto (-1)^{\rho(w)} w$ for $w \in \I$.
\end{lemma}

\begin{proof}
Let $\gamma$ be a nontrivial, pre-canonical $(\H,\I)$-structure, and define $\gamma' = \gamma[-1,-1]$. 
Let $(\psi,\I)$ be the unique pre-canonical structure on $\cA \I$ such that $\psi\( \rho_\gamma\(\overline{H}\) I\) = \rho_{\gamma}\(\overline{H}\) \psi(I)$ for $H \in \H$ and $I \in \cA \I$. 
Let $\psi' = D^{-1}\circ \psi \circ D$ where $D : \cA I \to \cA I$ is the $\cA$-linear map with $D(w) = (-1)^{\rho(w)} w$ for $w \in \I$. 
Since $\rho_{\gamma'}(H) = D^{-1}\circ \rho_{\gamma}(H) \circ D$ for $H \in \H$, it follows that $(\psi',\I)$ is a pre-canonical structure on $\cA \I$ 
such that 
\[\psi'\( \rho_{\gamma'}\(\overline{H}\) I\) = \rho_{\gamma'}\(\overline{H}\) \psi'(I)\qquad\text{for $H \in \H$ and $I \in \cA \I$}.\] Thus $\gamma'$ is pre-canonical. Moreover, the identity map $\cA \I \to \cA \I$ is evidently an isomorphism between the pre-canonical structures $(\psi,\I)$ and $(\psi',\I)$, with $D$ as a scaling factor.
\end{proof}

\begin{lemma}\label{l2-lem}
If $\gamma$ is a nontrivial, pre-canonical $(\H,\I)$-structure, then so is $\Theta(\gamma)$, and the (unique) associated pre-canonical structures on $\cA \I$ are strongly isomorphic via the identity map.
\end{lemma}

\begin{proof}
Let $\gamma$ be a nontrivial, pre-canonical $(\H,\I)$-structure, and define $\gamma' = \Theta(\gamma)$. 
Let $(\psi,\I)$ be the unique pre-canonical structure on $\cA \I$ such that $\psi\( \rho_\gamma\(\overline{H}\) I\) = \rho_{\gamma}\(\overline{H}\) \psi(I)$ for $H \in \H$ and $I \in \cA \I$. 
Then it also holds that 
$\psi\( \rho_{\gamma'}\(\overline{H}\) I\) = \rho_{\gamma'}\(\overline{H}\) \psi(I)$ for $H \in \H$ and $I \in \cA \I$
since $\rho_{\gamma'}(H) = \rho(\Theta(H))$ and $\overline{\Theta(H)} = \Theta(\overline{H})$.
Thus $\gamma'$ is also pre-canonical and its associated pre-canonical structure is strongly isomorphic to the one associated to $\gamma$.
\end{proof}

Before we can prove Theorem \ref{16-thm}, we require an additional lemma.
For this, let 
\[\cI,
\qquad \cI',
\qquad \cI'',
\qquand \cI'''\] be the free $\cA$-modules with  bases given by the symbols $I_w$,  $I'_w$, $I''_w$,  and $I'''_w$ respectively for $w \in \I$. View these as $\H$-modules relative to the $(\H,\I)$-structure $\Gamma$, $\Gamma'$, $\Gamma''$, and $\Gamma'''$ respectively.
Of course, $\cI$ defined in this way is   the same thing as $\cI$ defined by Theorem-Definition \ref{m2-thm}.
In addition, let $\epsilon $ denote the ring endomorphism of $\cA$ with $\epsilon(v) = -v$. 

\begin{lemma}  \label{isolem1}
There are unique pre-canonical $\H$-module structures on $\cI$, $\cI'$, $\cI''$, $\cI'''$, respectively, in which $\{I_w\}$, $\{I'_w\}$, $\{ I''_w\}$, $\{I'''_w\}$  indexed by $(\I,\leq)$ are the ``standard bases.''
Moreover, these pre-canonical structures are all isomorphic; the following maps are isomorphisms:
\begin{itemize}
\item[(a)] The $\cA$-linear map $\cI \to \cI'$ with $I_w \mapsto I'_w$ for $w \in \I$.

\item[(b)] The $\cA$-linear map $\cI'' \to \cI'''$ with $I''_w \mapsto I'''_w$ for $w \in \I$.

\item[(c)] The $\epsilon$-linear map $\cI\to \cI'''$ with $I_w \mapsto I'''_w$ for $w \in \I$.

\end{itemize}
Finally, the morphisms in (a), (b), (c) have as respective scaling factors the $\cA$-linear maps with 
\[ I_w \mapsto (-1)^{\rho(w)} I_w \qquand I''_w \mapsto (-1)^{\rho(w)} I''_w \qquand I_w \mapsto I_w\qquad\text{for }w \in \I.\]
\end{lemma}

 \begin{remark}
The ``bar involution'' of $\cI$ in the pre-canonical structure mentioned in this result is the one defined before Theorem-Definition \ref{m2-thm}.
One can show, though we omit the details here, that the   ``bar involutions'' of $\cI'$, $\cI''$, and $\cI'''$ are the respective $\cA$-antilinear maps
with
\[
{I'_{(x,\theta)} } \mapsto
  {H_x} \cdot I'_{(x^{-1},\theta)}
  \qquand
{I''_{(x,\theta)}} \mapsto   \overline{H_x} \cdot I''_{(x^{-1},\theta)}
\qquand
  {I'''_{(x,\theta)} } \mapsto
\sgn(x) \cdot {H_x} \cdot I'''_{(x^{-1},\theta)}
  \]
 for twisted involutions $(x,\theta) \in \I$.
\end{remark}

\begin{proof}
The uniqueness of the   pre-canonical $\H$-module structures is clear from Lemma \ref{uniquepre-lem}.
From Theorem-Definition \ref{precanonstruct2-thm} 
we already have a bar involution $I \mapsto \overline{I}$ on $\cI$ which forms a pre-canonical $\H$-module structure with $\{ I_w\}$ as the standard basis.
Define $r_{y,w} \in \cA$ for $y,w \in \I$  such that 
$ \overline{I_w} = \sum_{y \in \I} r_{y,w} I_y$.
In addition, for $x \in W$ and $y,z \in \I$ let $f^x_{y,z} \in \cA$ be such that 
$ H_x I_y = \sum_{z \in \I} f^x_{y,z} I_z$. 


Let $\cJ$ be the free $\cA$-module with a basis given by the symbols $J_w$ for $w \in \I$. View this as an $\H$-module relative to the $(\H,\I)$-structure $\gamma = \Gamma''[-1,-1] = \Theta(\Gamma''')$,
and define  $J \mapsto \overline{J}$  as the $\cA$-antilinear map $\cJ \to \cJ$ with 
$\overline{J_w} = \sum_{y \in \I} \epsilon(r_{y,w}) J_y$ for $w \in \I$.
It is immediate that this bar involution forms a pre-canoncal structure on $\cJ$ with $\{ J_w\}$ as the standard basis.
Since $H_s J_y =- \sum_{z \in \I} \epsilon(f^s_{y,z}) J_z$ for all $s \in S$ and $y \in \I$,
it follows moreover that 
$\overline{H_s J_y} = \overline{H_s}\cdot \overline{J_y}$,
which suffices to show
that $\overline {H} \cdot \overline{ J} = \overline{HJ}$ for all $H \in \H$ and $J \in \cJ$. We thus have a pre-canonical $\H$-module structure on $\cJ$. It is clear that the $\epsilon$-linear map $\cI \to \cJ$ with $I_w \mapsto J_w$ is an isomorphism of the pre-canonical structures on $\cI$ and $\cJ$, which has the identity map as a scaling factor.

One deduces the remaining  assertions in the lemma  from the existence of these isomorphic pre-canonical structures on $\cI$ and $\cJ$, using Lemmas \ref{l1-lem} and \ref{l2-lem} and the fact that 
\[ \Gamma' = \Theta(\Gamma)[-1,-1]\qquand \Gamma'' = \gamma[-1,-1]\qquand \Gamma''' = \Theta(\gamma).\]
%
\end{proof}

\begin{proof}[Proof of Theorem \ref{16-thm}]
Let $\gamma$ be a nontrivial $(\H,\I)$-structure which is pre-canonical, and write $\psi : \cA \I \to \cA \I$ for the associated bar involution. We claim  that $\gamma_{11}$ and $\gamma_{31}$ must then belong to $\ZZ[v+v^{-1}]$. To see this let $\theta \in \Aut(W,S)$ be an involution and let $s \in S$. If $s \neq \theta(s)$ then  $w = (s\cdot \theta(s),\theta) \in \I$ and we have
\[ \overline{\gamma_{11}}\cdot \psi(w) + \overline{\gamma_{12}} \cdot \theta = \psi(\gamma(H_s)\theta) = \gamma(H_s + v^{-1}-v) \theta = \gamma_{11} \cdot w + (\gamma_{12}+v^{-1}-v)\cdot \theta.\]
On the other hand if $s = \theta(s)$ then $w = (s,\theta) \in \I$ and we have
\[ \overline{\gamma_{31}}\cdot \psi(w) + \overline{\gamma_{32}} \cdot \theta = \psi(\gamma(H_s)\theta) = \gamma(H_s + v^{-1}-v) \theta = \gamma_{31} \cdot w + (\gamma_{32}+v^{-1}-v)\cdot \theta.\]
These equations, compared with the unitriangular property of the bar involution,
imply  $\overline{\gamma_{11}} = \gamma_{11}$
and 
$\overline{\gamma_{31}} = \gamma_{31}$; hence these two parameters must belong to $\ZZ[v+v^{-1}]$ as claimed.
Since Theorem \ref{HI-thm} implies that 
\[ \gamma_{11} \cdot \gamma_{21} = 1 \qquand \gamma_{31}\cdot \gamma_{41} \in \{ \pm (v-v^{-1})\}\]
it necessarily follows that $\gamma_{11} ,\gamma_{31} \in \{ \pm 1\}$.
From Theorem \ref{HI-thm} we conclude that for some $\varepsilon_i \in \{ \pm 1\}$ 
we have 
$\gamma[\varepsilon_1,\varepsilon_2] \in \{ \Gamma,\Gamma',\Gamma'',\Gamma'''\}$.
Thus $\gamma$ must be one of 16 different $(\H,\I)$-structures.
It is a simple
exercise to show that $\gamma$ is pre-canonical if and only if  $\gamma[\varepsilon_1,\varepsilon_2]$ is pre-canonical; moreover, the associated pre-canonical structures are isomorphic.
Hence,  by Lemma \ref{isolem1} we conclude that all 16 possibilities for $\gamma$ are pre-canonical, and that the associated pre-canonical structures are all isomorphic to the one in Theorem-Definition \ref{precanonstruct2-thm}.
\end{proof}

%
%
%

Finally, we return to
Theorems \ref{H2I-thm} and \ref{32-thm}. These results follow by  arguments similar to the ones just given, and so we only sketch the main ideas to their proofs.

\begin{proof}[Sketch of proof of Theorem \ref{H2I-thm}]
The result follows by nearly the same argument as in the proof Theorem \ref{HI-thm}, using three lemmas analogous to Lemmas \ref{diag-lem}, \ref{theta-lem}, and \ref{eqlem1}, \emph{mutatis mutandis}. We omit the details.
\end{proof}

\begin{proof}[Sketch of proof of Theorem \ref{32-thm}]
One deduces
that at most 32 nontrivial $(\H_2,\I)$-structures are pre-canonical exactly as in the proof of Theorem \ref{16-thm}: first argue that any such structure $\gamma$ has $\overline{\gamma_{11}} = \gamma_{11}$ and $\overline{\gamma_{31}} = \gamma_{31}$, and then appeal to Theorem \ref{H2I-thm}. 
The claim that these $(\H_2,\I)$-structures are in fact all pre-canonical,
along with the second sentence in the theorem,
follows from Lemmas \ref{l1-lem} and \ref{l2-lem}, which hold
 \emph{mutatis mutandis} with ``$(\H,\I)$-structure'' replaced by $(\H_2,\I)$-structure'' and $\Theta$ replaced by a slightly different involution on $4\times 2$ matrices.
\end{proof}

\subsection{Application to inversion formulas}
\label{app-sect}

In this last section, we use the lemmas in the previous section to prove an inversion formula for the canonical bases introduced in Section \ref{exist-sect}.
Let $V$ be a free $\cA$-module of finite rank, with a pre-canonical structure $(\psi, \{ a_c\})$, the standard basis indexed by $(C,\leq)$.
Define $V^*$ as the set of $\cA$-linear maps $V \to \cA$. This is naturally a free $\cA$-module: a basis is given by the $\cA$-linear maps 
$a_c^* : V \to V$ for $c \in C$ defined by
\[a_c^*(a_{c'}) = \delta_{c,c'}\qquad\text{for }c' \in C.\]
Define 
$ \psi^* : V ^* \to V^*$
as the $\cA$-antilinear map such that 
\[ \psi^*(f) (v) = \overline{f\circ \psi(v)}\qquad\text{for }f \in V^*\text{ and }v \in V.\]
Also let $\leq^\op$ denote the partial order on $C$ with $c \leq^\op c'$ if and only if $c' \leq c$.
The following appears in a slightly more general form as \cite[Proposition 7.1]{Webster}.

\begin{proposition}[Webster \cite{Webster}] \label{delta-prop} The ``bar involution'' $\psi^*$ and ``standard basis'' $\{ a^*_c\}$, indexed by the partially ordered set $(C,\leq^\op)$, form a pre-canonical structure on $V^*$.
 If  $V$  has a canonical basis $\{b_c\}$, then the dual basis $\{ b_c^*\}$ of $V^*$ is canonical relative to $(\psi^*,\{a_c^*\})$.
\end{proposition}

%
%
%

Let $\cB$ denote a free $\cA$-algebra with a pre-canonical structure; write $\overline b$ for the image of $b \in \cB$ under the corresponding bar involution.
Suppose $V$ is a $\cB$-module and $(\psi,\{a_c\})$ is a pre-canonical $\cB$-module structure. 
Assume $\cB$ has a distinguished $\cA$-algebra antiautomorphism $b \mapsto b^\dag$. We may then view $V^*$ as a $\cB$-module by defining $b  f$ for $b \in \cB$ and $f \in V^*$ to be the map with the formula
\be\label{make-eq} (b   f) (v) = f(b^\dag v)\qquad\text{for } v \in V.\ee

\begin{proposition}\label{dag-prop}
Suppose the maps $b \mapsto b^\dag$  and  $b \mapsto \overline{b}$ commute. Then the pre-canonical structure $(\psi^*,\{a^*_c\})$ on $V^*$ is a pre-canonical $\cB$-module structure.
\end{proposition}

\begin{proof}
One just needs to check that if $b \in \cB$ and $f \in V^*$ then $\psi^*(bf) = \overline{b} \cdot \psi^*(f)$, and this is straightforward from the commutativity hypothesis in the proposition.
%
\end{proof}

Assume $(W,S)$ is a finite Coxeter system, so that $W$ has 
a longest element $w_0$. Recall that since the longest element is unique, we have $w_0 = w_0^{-1} =\theta(w_0)$ for all $\theta \in \Aut(W,S)$. Write $\theta_0$ for  the inner automorphism of $W$ given by $w \mapsto w_0 ww_0$. This map is an automorphism of the poset $(W,\leq )$ 
and in particular is length-preserving \cite[Proposition 2.3.4(ii)]{CCG}; thus it belongs to $\Aut(W,S)$. 
In fact, $\theta_0$ lies in the center of $\Aut(W,S)$.
Let $w_0^+= (w_0,\theta_0) \in W^+$.
Observe that $w_0^+$ is a central involution in $W^+$,
and so if
 $w = (x,\theta) \in \I$ then $ww_0^+ = (xw_0,\theta\theta_0) \in \I$.

We may use the results in the previous sections to prove an inversion formula for the structure constants of the canonical bases of $\cL$, $\cL'$, and $\cI$ given in Theorem-Definitions \ref{precanonstruct1-thm}, \ref{preL'-thm}, and \ref{precanonstruct2-thm}.

\begin{theorem}
Let $F \in \{ \pi, \pi', \iota\}$.
Then 
$\sum_{w \in \I} (-1)^{\rho(x)+\rho(w)} \cdot  F_{x,w}\cdot  F_{yw_0^+,ww_0^+} =\delta_{x,y}$ for  $x,y \in \I$.

%
%
\end{theorem}

Lusztig proves the version of this statement with $F = \pi$ as \cite[Theorem 7.7]{LV2}.

\begin{proof}
We only consider the case $F = \iota$ as the argument in the other cases is similar.
There is a unique antiautomorphism $ H \mapsto H^\dag$  of $\H$ with $H_w \mapsto H_{w^{-1}}$ for $w \in W$.
We make $\cI^*$ into an $\H$-module relative to this antiautomorphism via the formula \eqref{make-eq}.
Let $s \in S$ and $w \in \I$.
Since $w_0^+$ is central, we have $sw = ws$ if and only if $sww_0^+ = ww_0^+s$.
Since $x\leq y$ if and only if $yw_0 \leq xw_0$ for any $x,y \in W$ (see \cite[Proposition 2.3.4(i)]{CCG}),
it follows that $sw < w$ if and only if $sww_0^+ > sww_0^+$, and also that $\rho(xw_0^+) - \rho(yw_0^+) = \rho(y) - \rho(x)$ for $x,y \in \I$.
Given these facts it is straightforward to check that if $\cI'$ is the $\H$-module defined before Lemma \ref{isolem1}, then the $\cA$-linear map $\varphi: \cI' \to \cI^*$ with $\varphi(I_{w}' )= I^*_{ww_0^+}$ for $w \in \I$ is an isomorphism of $\H$-modules.

We have a pre-canonical $\H$-module structure on $\cI'$ from Lemma \ref{isolem1}.
Likewise, since the maps $H \mapsto H^\dag$ and $H\mapsto \overline{H}$ commute, we have a pre-canonical $\H$-module structure on $\cI^*$ from Proposition \ref{dag-prop}. 
Write $\psi^*$ for the bar involution of $\cI^*$ in this structure.
Then $(\varphi^{-1}\circ \psi^* \circ \varphi, \{ I'_w\})$ is  another pre-canonical $\H$-module structure 
on $\cI'$, so the uniqueness assertion in Lemma \ref{isolem1} implies  that $\varphi^{-1} \circ\psi^* \circ \varphi$ is equal to the bar involution $I \mapsto \overline{I}$ on $\cI'$, and 
thus
  $\varphi$ is a strong isomorphism between the pre-canonical structures on $\cI'$ and $\cI^*$. Composing $\varphi$ with the map  in Lemma \ref{isolem1}(a), it follows that
the $\cA$-linear map $\cI \to \cI^*$ with $I_w \mapsto I_{ww_0^+}^*$ 
is an isomorphism of pre-canonical structures (though not of $\H$-modules), having as a scaling factor the $\cA$-linear map $D : \cI \to \cI$ with $D(I_w) = (-1)^{\rho(w)} I_w$ for $w \in \I$.

From 
Proposition \ref{canon-prop}, we deduce that  elements of 
the canonical basis $\{ \underline I_w^*\}$ of $\cI^*$  have the form
$ \underline I_y^* = I_y^* + \sum_{w>y} (-1)^{\rho(y) - \rho(w)}  \iota_{yw_0^+,ww_0^+}\cdot I_w^*$.
Since $\underline I_y^*(\underline I_x)  =\delta_{x,y}$ for $x,y \in \I$ by Proposition \ref{delta-prop},
we deduce that $MN = 1$ where $M$ and $N$ are the $(\I\times \I)$-indexed   matrices with $M_{y,w} = (-1)^{\rho(y) - \rho(w)}\iota_{ww_0^+,yw_0^+}$ and $N_{w,x} = \iota_{w,x}$. Since $M$ and $N$ are finite square matrices, $MN= 1$ implies $NM=1$; the desired inversion formula is equivalent to the second equality.
\end{proof}

%
%
%
%
%
%
%

\end{document}